\documentclass[12pt,a4paper]{article}

\usepackage{amsfonts,amssymb}
\usepackage{amsthm,amsmath,amstext}
\usepackage{cancel,xspace}
\usepackage{epsfig,fancybox,microtype}
\usepackage{graphicx,mathrsfs,ifpdf}
\usepackage[justification=centering,labelformat=empty,labelsep=none]{subfig}
\usepackage{tikz}
\usetikzlibrary{shapes}

\newtheorem{theorem}{Theorem}
\newtheorem{lemma}[theorem]{Lemma}
\newtheorem{corollary}[theorem]{Corollary}

\ifpdf
\usepackage[pdftex]{hyperref}
\else
\newcommand{\href}[2]{#2}
\fi

\newcommand{\abs}[1]{\left\lvert#1\right\rvert}
\DeclareMathOperator{\ch}{ch}
\DeclareMathOperator{\sge}{\geqslant\!}
\DeclareMathOperator{\sle}{\leqslant\!}

\renewcommand{\ge}{\geqslant}
\renewcommand{\le}{\leqslant}
\newcommand{\col}[1]{edge-face $ #1 $-colouring\xspace}

\newenvironment{tikzgraph}
  {\begin{tikzpicture}
      [vertex/.style={circle, draw=black, fill=white, inner sep=0.5pt, minimum
        size=8pt},
       2-vertex/.style={circle, draw=black, fill=black, inner sep=0mm, minimum
        size=6pt},%
       3-vertex/.style={regular polygon,regular polygon sides=3, draw=black, fill=black, inner sep=0mm, minimum
        size=8pt},%
       4-vertex/.style={regular polygon,regular polygon sides=4, draw=black, fill=black, inner sep=0mm, minimum
        size=8pt},%
       edge/.style={semithick},%
       font=\scriptsize
      ]\begin{scope}}
  {\end{scope}\end{tikzpicture}}

\newcommand{\unit}{20pt}
\def \smallface {(0,0) .. controls (0.2*\unit,1.0*\unit) and (1.532*\unit,1.0*\unit) .. (1.732*\unit,0);}

\begin{document}
\title{Every plane graph of maximum degree $8$ has an edge-face $9$-colouring}
\author{Ross J. Kang\thanks{%
School of Engineering and Computing Sciences, Durham University, UK ({\tt ross.kang@gmail.com}). This research was conducted while this author was a Postdoctoral Fellow at McGill University, supported by the {\em National Sciences and Engineering Research Council of Canada} (NSERC).  He is currently supported by the {\em Engineering and Physical Sciences Research Council} (EPSRC), grant EP/G066604/1.}
\and Jean-S\'ebastien Sereni\thanks{%
CNRS (LIAFA, Universit\'e Denis Diderot), Paris, France
and Department of Applied Mathematics (KAM), Faculty of Mathematics and
Physics, Charles University, Prague, Czech Republic
({\tt sereni@kam.mff.cuni.cz}). This author's work was partially
supported by the French \emph{Agence Nationale de la Recherche} under reference
\textsc{anr 10 jcjc 0204 01}.}
\and Mat\v ej Stehl\'ik\thanks{%
UJF-Grenoble 1 / CNRS / Grenoble-INP, G-SCOP UMR5272 Grenoble, F-38031, France
({\tt matej.stehlik@g-scop.inpg.fr}).}
}
\date{\today}
\maketitle

\begin{abstract}
An edge-face colouring of a plane graph with edge set $E$ and face set $F$ is a colouring of the elements of $E \cup F$ such that adjacent or incident elements receive different colours. Borodin proved that every plane graph of maximum degree $\Delta\ge10$ can be edge-face coloured with $\Delta+1$ colours. Borodin's bound was recently extended to the case where $\Delta=9$. In this paper, we extend it to the case $\Delta=8$.
\end{abstract}


\section{Introduction}
\label{sec:intro}

Let $G$ be a plane graph with vertex set $V$, edge set $E$ and face set $F$.
Given a positive integer $k$, an \emph{\col{k} of $G$} is a mapping
$\lambda:E\cup F\rightarrow\{1,2,\ldots,k\}$ such that
\begin{enumerate}
\item
$\lambda(e)\neq \lambda(e')$ for every pair $(e,e')$ of adjacent edges;
\item
$\lambda(e)\neq \lambda(f)$ for edge $e$ and every face $f$ incident to $e$;
\item
$\lambda(f)\neq \lambda(f')$ for every pair $(f,f')$ of adjacent faces
with $f\neq f'$.
\end{enumerate}
The requirement in $(iii)$ that $f$ and $f'$ be distinct is only
relevant for graphs containing a cut-edge; such graphs would not
have an edge-face colouring otherwise. Let $\chi_{ef}(G)$ be the
value of the smallest integer $k$ such that there exists an \col{k} of $G$.  Although this problem is well-defined for graphs with loops or multiple edges, we shall throughout the paper only consider graphs that are simple (and this requirement is necessary, for instance, in Lemma~\ref{smalladj}).  We comment here that the multigraph formed by replacing each edge in a triangle by $\Delta/2$ parallel edges is planar with maximum degree $\Delta$ and requires at least $3\Delta/2$ colours in an edge colouring (let alone an edge-face colouring).

Edge-face colourings were first studied by Jucovi\v{c}~\cite{Juc69}
and Fiam\v{c}\'{\i}k~\cite{Fia71}, who considered $3$- and $4$-regular
graphs.  A conjecture of Mel'nikov~\cite{Mel75} spurred research into
upper bounds on $\chi_{ef}(G)$ for plane graphs $G$ with $\Delta(G) \le
\Delta$.  For small values of $\Delta$, the best bounds known are
$\Delta + 3$ for $\Delta \in \{2,\ldots,6\}$~\cite{Bor87,SaZh97,Wal97}
and $\Delta + 2$ for $\Delta = 7$~\cite{SaZh01b}.  For $\Delta \ge 10$,
Borodin~\cite{Bor94} proved the bound of $\Delta+1$. This is tight, as
can be seen by considering trees. Recently, the second and third
authors~\cite{SeSt11} extended the $\Delta + 1$ bound to the case
$\Delta = 9$ by proving that every plane graph of maximum degree $9$ has
an \col{10}. Here, we settle the case $\Delta = 8$.
\begin{theorem}
\label{thm:main}
Every plane graph of maximum degree $8$ has an \col{9}.
\end{theorem}
\noindent
Our result is a strengthening of the $\Delta = 7$ result of Sanders and Zhao~\cite{SaZh01b}; it can also be viewed as an extension of the recent work on $\Delta = 9$~\cite{SeSt11}.
The problem of finding the provably optimal upper bounds on
$\chi_{ef}(G)$ for plane graphs $G$ with $\Delta(G) \le \Delta$ remains open for $\Delta \in \{4,5,6,7\}$.

We prove Theorem~\ref{thm:main} by contradiction.
From now on, we let $G=(V,E,F)$ be a counter-example to the statement
of Theorem~\ref{thm:main} with as few edges as possible. That is, $G$ is
a plane graph of maximum degree $8$ and no \col{9}, but every plane
graph of maximum degree at most $8$ with less than $|E|$ edges has an \col{9}.
In particular, for every edge $e\in E$ the plane
subgraph $G-e$ of $G$ has an \col{9}.
First, we describe various structural properties of $G$ in
Section~\ref{sec:reddesc}; the proofs of these properties are given at the
end of this paper in Section~\ref{sec:redproof}. In Section~\ref{sec:disdesc}
we describe the discharging rules. In Section~\ref{sec:disproof} we use the
discharging rules and the structural properties of $G$ to obtain a contradiction,
and thus a proof of Theorem~\ref{thm:main}.

Our discharging procedure was developed through several rounds, with corrective adjustments and optimisations included in each round, starting from a na{\"i}ve scheme in which only the vertices of degree at least $7$ compensated for the deficit of charge on triangles.  
A breakthrough in the design of our strategy was the
realisation that Lemma~\ref{lem-new} below could allow us
to conserve considerable charge at degree $7$ or $8$ vertices incident to
faces of a particular type. We could then balance these
savings against the loss of charge to incident triangles with the
development of further reducible configurations.
As will become apparent, the analysis of the final charge of vertices of degree $7$ or $8$ is
particularly involved.

In the sequel, a vertex of degree $d$ is
called a \emph{$d$-vertex}. A vertex is an $(\sle d)$\emph{-vertex}
if its degree is at most $d$; it is an \emph{$(\sge d)$-vertex} if its
degree is at least $d$.
The notions of \emph{$d$-face}, \emph{$(\sle d)$-face}
and \emph{$(\sge d)$-face} are defined
analogously as for the vertices, where the \emph{degree} of a face is
the number of edges incident to it.
A face of length $3$ is called a \emph{triangle}.
For integers $a,b,c$, an \emph{$(\sle a,\sle b,\sle c)$-triangle} is a
triangle $xyz$ of $G$ with $\deg(x)\le a$, $\deg(y)\le b$ and $\deg(z)\le c$.
The notions of $(a,\sle b,\sle c)$-triangles, $(a, b,\sge c)$-triangles,
$(a,\sle b,c,d)$-faces, and so on, are defined analogously. A vertex is
\emph{triangulated} if all its incident faces are triangles.


\section{Reducible configurations}
\label{sec:reddesc}

For our proof of Theorem~\ref{thm:main}, we identify that some plane graphs are
\emph{reducible configurations}, i.e.~configurations that cannot be part of the
chosen embedding of $G$. Their reducibility follows from Lemmas~\ref{lem-new}--\ref{lem.E3}; these lemmas are proved in Section~\ref{sec:redproof}.  In this section, we give an explicit description of the reducible configurations as well as the statement of Lemma~\ref{lem-new}.

For convenience, we depict these configurations in Figure~\ref{fig:configs}.
We use the following notational conventions for vertices:
$2$-, $3$- and $4$-vertices are depicted by black bullets, black triangles and black squares, respectively;
a white bullet containing a number represents a vertex of degree that quantity;
an empty white bullet represents a vertex of arbitrary degree (but at least that shown in the figure).
For faces, we use the following conventions:
a straight line indicates a single edge;
a curved line indicates a portion of the face with an unspecified number of edges;
a curved face that is shaded grey represents an $(\sle4)$-face.

The following configurations are reducible. 
Note that, for any of the below, if an edge can be removed without affecting the prescribed incidence or facial structure, then the configuration remains reducible; for example, B6 modified by replacing the $6$ by a $5$ or $4$ is reducible.

\begin{itemize}
\item[\textbf{A0}] A $1$-vertex.
\end{itemize}

\subsubsection*{Configurations with faces incident to a $2$-vertex}

\begin{itemize}
\item[\textbf{A1}]  A triangle incident to a $2$-vertex.

\item[\textbf{A2}]  A $4$-face incident to a $2$-vertex and an $(\sle3)$-vertex.

\item[\textbf{A3}]  A $2$-vertex adjacent to a $3$-vertex and an $(\sle5)$-vertex.
\end{itemize}

\subsubsection*{Configurations with an edge incident to a $(\sle4)$-face}

\begin{itemize}
\item[\textbf{B1}]  An edge $uv$ that is incident to an $(\sle4)$-face, with $\deg(u) + \deg(v) \le 9$.

\item[\textbf{B2}]  A triangle $uvw$ with $\deg(u) + \deg(v) \le 10$ and $\deg(w) = 6$.

\item[\textbf{B3}]  A triangle $uvw$ with $uw$ incident to two $(\sle4)$-faces, and $\deg(u) + \deg(v) \le 10$ and $\deg(w) = 7$.

\item[\textbf{B4}]  A triangle $uvw$ with $uw$ adjacent to two $(\sle4)$-faces, $vw$ incident to two $(\sle4)$-faces, and $\deg(u) + \deg(v) \le 10$.
\end{itemize}

\subsubsection*{Configurations with an edge incident to two $(\sle4)$-faces}

\begin{itemize}
\item[\textbf{C1}]  An edge $uv$ that is incident to two $(\sle4)$-faces, with $\deg(u) + \deg(v) \le 10$.

\item[\textbf{C2}]  A triangle $uvw$ with $uv$ incident to two $(\sle4)$-faces, and $\deg(u) + \deg(v) \le 11$ and $\deg(w) = 6$.

\item[\textbf{C3}]  A triangle $uvw$ with $uv$ and $uw$ each incident to two $(\sle4)$-faces,
and $\deg(u) + \deg(v) \le 11$ and $\deg(w) = 7$.

\item[\textbf{C4}]  A triangle $uvw$ with $vw$ incident to the triangle $vwx$ and $wx$ incident to two $(\sle4)$-faces, and $\deg(u) = \deg(x) = 3$. 

\item[\textbf{C5}]  A triangle $uvw$ with $vw$ incident to the triangle $vwx$ and $wx$ incident to two $(\sle4)$-faces, and $\deg(u) + \deg(v) \le 10$ and $\deg(v) + \deg(x) \le 11$.
\end{itemize}

\subsubsection*{Configurations along a $2$-path}

\begin{itemize}
\item[\textbf{D1}]  A $2$-path $uvw$ such that $vwx$ is a triangle, with $uv$ incident to an $(\sle4)$-face,
$vw$ and $vx$ each incident to two $(\sle4)$-faces,
and $\deg(u) + \deg(v) \le 10$ and $\deg(v) + \deg(w) \le 11$.

\item[\textbf{D2}]  A $2$-path $uvw$ such that $vwx$ is a triangle,
with $uv$, $vw$ and $vx$ each incident to two $(\sle4)$-faces,
and $\deg(u) + \deg(v) \le 11$ and $\deg(v) + \deg(w) \le 11$.

\item[\textbf{D3}]  A $2$-path $uvw$ such that $vwx$ is a triangle, with $vx$ incident to two $(\sle4)$-faces, and $\deg(u) = 2$, $\deg(v) = 7$ and $\deg(u) = 3$.

\item[\textbf{D4}]  A $2$-path $uvw$ such that $vwx$ is a triangle,
with $vw$ and $vx$ each incident to two $(\sle4)$-faces,
and $\deg(u) = 2$, $\deg(v) = 7$ and $\deg(u) = 4$.
\end{itemize}

\noindent \emph{Note on configurations D1 and D2.}  An $(\sle4)$-face incident to $uv$ is not ruled out from also being an $(\sle4)$-face (distinct from $vwx$) incident to $vw$ or $vx$.  In this sense, the figures representing configurations D1 and D2 in Figure~\ref{fig:configs} belie the configurations' fuller forms.

\subsubsection*{Exceptional configurations}

\begin{itemize}
\item[\textbf{E1}]  A $4$-path $uvwxy$, such that $uvz$, $vwz$, $wxz$ and $xyz$ are triangles, with $yz$ incident to two $(\sle4)$-faces, and $\deg(v) = 3$, $\deg(x) = 4$.

\item[\textbf{E2}]  A $4$-path $uvwxy$, such that $uvz$, $vwz$, $wxz$ and $xyz$ are triangles, and $\deg(v) = 3$, $\deg(x) = 4$ and $\deg(y) = 6$.

\item[\textbf{E3}]  A triangulated $8$-vertex that is adjacent to both a $3$-vertex and a $4$-vertex.

\item[\textbf{E4}]  A $3$-path $uvwx$, with $uv$ incident to an $(\sle4)$-face, and $\deg(u)\le 5$, $\deg(v)=6$, $\deg(w)=2$ and $\deg(x)=3$.
\end{itemize}

\subsubsection*{Special lemmas for $(\sge5)$-faces.}

An edge $uv$ is \emph{loose} if $\deg(u)+\deg(v)\le8$.
The following lemma implies a general set of reducible configurations for $(\sge5)$-faces.  These configurations are not depicted in Figure~\ref{fig:configs}.
\begin{lemma}\label{lem-new}
Let $f$ be a $d$-face of $G$ incident to $x$ loose edges and $q$ vertices of degree $2$.
If $d\ge5$ and $x\ge1$, then $2d-q-x\ge9$.
\end{lemma}
\noindent 
For $(\sge6)$-faces, we require one more configuration not depicted in Figure~\ref{fig:configs}.
\begin{lemma}\label{lem-newnew}
Let $u$ and $v$ be two adjacent $2$-vertices in $G$.
If $u' \ne v$ and $v'\ne u$ are neighbours of $u$ and $v$, respectively, then $u' = v'$ and $\deg(u')=8$.
\end{lemma}

\begin{figure}
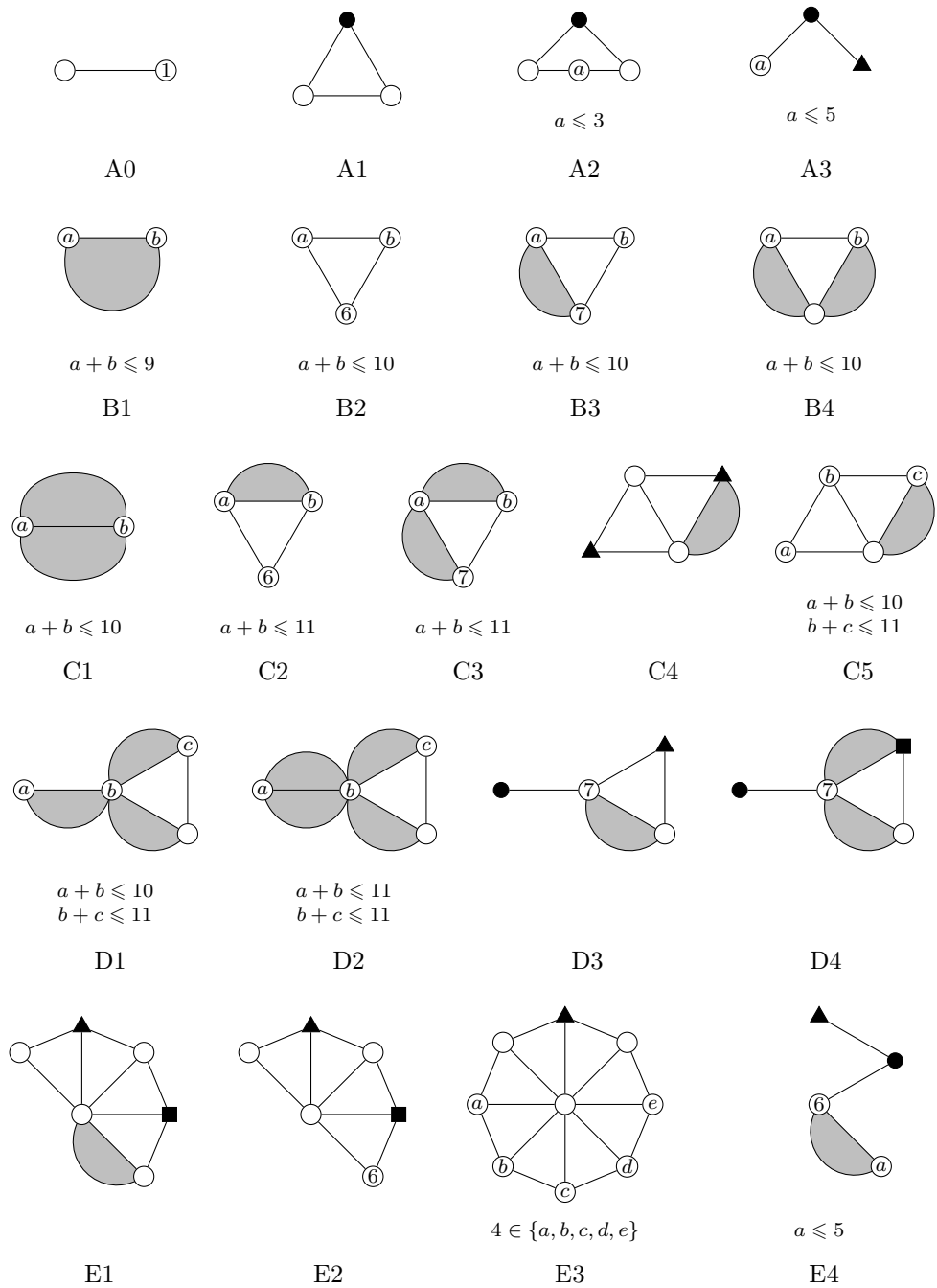
%
\centering
\hspace*{\fill}
\subfloat[A0]{%
  \begin{tikzgraph}
    \draw (-\unit,0) node[vertex] (a0) {};
    \draw (\unit,0) node[vertex] (a1) {1};
    \draw (a0) edge (a1);
    \useasboundingbox (-1.5*\unit,-1.5*\unit) rectangle (1.5*\unit,1.5*\unit);
  \end{tikzgraph}
  }
\hspace*{\fill}
\subfloat[A1]{%
  \begin{tikzgraph}
    \useasboundingbox (-1.5*\unit,-1.5*\unit) rectangle (1.5*\unit,1.5*\unit);
    \draw (90:\unit) node[2-vertex] (a0) {};
    \foreach\i in {1,...,2}
    {
      \draw (90+120*\i:\unit) node[vertex] (a\i) {};
    }
    \draw (a0) edge (a1) (a1) edge (a2) (a2) edge (a0);
  \end{tikzgraph}
}
\hspace*{\fill}
\subfloat[A2]{%
  \begin{tikzgraph}
    \useasboundingbox (-1.5*\unit,-1.5*\unit) rectangle (1.5*\unit,1.5*\unit);
    \draw (0:\unit) node[vertex] (a0) {};
    \draw (90:\unit) node[2-vertex] (a1) {};
    \draw (180:\unit) node[vertex] (a2) {};
    \draw (0:0) node[vertex] (a3) {$a$};
    \draw (a0) edge (a1) (a1) edge (a2) (a2) edge (a3) (a3) edge (a0);
    \draw (-90:\unit) node {$a\le3$};
  \end{tikzgraph}
}
\hspace*{\fill}
\subfloat[A3]{%
  \begin{tikzgraph}
    \useasboundingbox (-1.5*\unit,-1.6*\unit) rectangle (1.5*\unit,1.4*\unit);
    \draw (180:\unit) node[vertex] (a0) {$a$};
    \draw (90:\unit) node[2-vertex] (a1) {};
    \draw (0:\unit) node[3-vertex] (a2) {};
    \draw (a0) edge (a1);
    \draw (a2) edge (a1);
    \draw (-90:\unit) node {$a\le5$};
  \end{tikzgraph}
}
\hspace*{\fill}
\\
\vspace{0.25cm}
\hspace*{\fill}
\subfloat[B1]{%
  \begin{tikzgraph}
    \useasboundingbox (-1.6*\unit,-2.4*\unit) rectangle (1.6*\unit,1.1*\unit);
    \draw[fill=gray!50] (30:\unit) .. controls (-45:2*\unit) and (-135:2*\unit) .. (150:\unit);
    \draw (30:\unit) node[vertex] (a0) {$b$};
    \draw (150:\unit) node[vertex] (a1) {$a$};
    \draw (a0) edge (a1);
    \draw (-90:2*\unit) node {$a+b\le9$};
  \end{tikzgraph}
}
\hspace*{\fill}
\subfloat[B2]{%
  \begin{tikzgraph}
    \useasboundingbox (-1.6*\unit,-2.4*\unit) rectangle (1.6*\unit,1.1*\unit);
    \draw (30:\unit) node[vertex] (a0) {$b$};
    \draw (150:\unit) node[vertex] (a1) {$a$};
    \draw (-90:\unit) node[vertex] (a2) {$6$};
    \draw (a0) edge (a1) (a1) edge (a2) (a2) edge (a0);
    \draw (-90:2*\unit) node {$a+b\le10$};
  \end{tikzgraph}
}
\hspace*{\fill}
\subfloat[B3]{%
  \begin{tikzgraph}
    \useasboundingbox (-1.6*\unit,-2.4*\unit) rectangle (1.6*\unit,1.1*\unit);
    \draw[shift={(270:\unit)},rotate=120,fill=gray!50] \smallface;
    \draw (30:\unit) node[vertex] (a0) {$b$};
    \draw (150:\unit) node[vertex] (a1) {$a$};
    \draw (-90:\unit) node[vertex] (a2) {$7$};
    \draw (a0) edge (a1) (a1) edge (a2) (a2) edge (a0);
    \draw (-90:2*\unit) node {$a+b\le10$};
  \end{tikzgraph}
}
\hspace*{\fill}
\subfloat[B4]{%
  \begin{tikzgraph}
    \useasboundingbox (-1.6*\unit,-2.4*\unit) rectangle (1.6*\unit,1.1*\unit);
    \draw[shift={(270:\unit)},rotate=120,fill=gray!50] \smallface;
    \draw[shift={(30:\unit)},rotate=240,fill=gray!50] \smallface;
    \draw (30:\unit) node[vertex] (a0) {$b$};
    \draw (150:\unit) node[vertex] (a1) {$a$};
    \draw (-90:\unit) node[vertex] (a2) {};
    \draw (a0) edge (a1) (a1) edge (a2) (a2) edge (a0);
    \draw (-90:2*\unit) node {$a+b\le10$};
  \end{tikzgraph}
}
\hspace*{\fill}
\\
\vspace{0.25cm}
\hspace*{\fill}
\subfloat[C1]{%
  \begin{tikzgraph}
    \useasboundingbox (-1.6*\unit,-2.4*\unit) rectangle (1.6*\unit,1.6*\unit);
    \draw[fill=gray!50] (0:\unit) .. controls (45:2*\unit) and (135:2*\unit) .. (180:\unit);
    \draw[fill=gray!50] (0:\unit) .. controls (-45:2*\unit) and (-135:2*\unit) .. (180:\unit);
    \draw (0:\unit) node[vertex] (a0) {$b$};
    \draw (180:\unit) node[vertex] (a1) {$a$};
    \draw (a0) edge (a1);
    \draw (-90:2*\unit) node {$a+b\le10$};
  \end{tikzgraph}
}
\hspace*{\fill}
\subfloat[C2]{%
  \begin{tikzgraph}
    \useasboundingbox (-1.6*\unit,-2.4*\unit) rectangle (1.6*\unit,1.6*\unit);
    \draw[shift={(150:\unit)},fill=gray!50] \smallface;
    \draw (30:\unit) node[vertex] (a0) {$b$};
    \draw (150:\unit) node[vertex] (a1) {$a$};
    \draw (-90:\unit) node[vertex] (a2) {$6$};
    \draw (a0) edge (a1) (a1) edge (a2) (a2) edge (a0);
    \draw (-90:2*\unit) node {$a+b\le11$};
  \end{tikzgraph}
}
\hspace*{\fill}
\subfloat[C3]{%
  \begin{tikzgraph}
    \useasboundingbox (-1.6*\unit,-2.4*\unit) rectangle (1.6*\unit,1.6*\unit);
    \draw[shift={(150:\unit)},fill=gray!50] \smallface;
    \draw[shift={(270:\unit)},rotate=120,fill=gray!50] \smallface;
    \draw (30:\unit) node[vertex] (a0) {$b$};
    \draw (150:\unit) node[vertex] (a1) {$a$};
    \draw (-90:\unit) node[vertex] (a2) {$7$};
    \draw (a0) edge (a1) (a1) edge (a2) (a2) edge (a0);
    \draw (-90:2*\unit) node {$a+b\le11$};
  \end{tikzgraph}
}
\hspace*{\fill}
\subfloat[C4]{%
  \begin{tikzgraph}
    \useasboundingbox (-2*\unit,-2.9*\unit) rectangle (1.2*\unit,1.1*\unit);
    \draw[shift={(30:\unit)},rotate=-120,fill=gray!50] \smallface;
    \draw (30:\unit) node[3-vertex] (a0) {};
    \draw (150:\unit) node[vertex] (a1) {};
    \draw (-90:\unit) node[vertex] (a2) {};
    \draw (210:2*\unit) node[3-vertex] (a3) {};
    \draw (a0) edge (a1) (a1) edge (a2) (a2) edge (a0) (a1) edge (a3) (a2) edge (a3);
  \end{tikzgraph}
}
\hspace*{\fill}
\subfloat[C5]{%
  \begin{tikzgraph}
    \useasboundingbox (-2*\unit,-2.9*\unit) rectangle (1.2*\unit,1.1*\unit);
    \draw[shift={(30:\unit)},rotate=-120,fill=gray!50] \smallface;
    \draw (30:\unit) node[vertex] (a0) {$c$};
    \draw (150:\unit) node[vertex] (a1) {$b$};
    \draw (-90:\unit) node[vertex] (a2) {};
    \draw (210:2*\unit) node[vertex] (a3) {$a$};
    \draw (a0) edge (a1) (a1) edge (a2) (a2) edge (a0) (a1) edge (a3) (a2) edge (a3);
    \draw (-0.375*\unit,-2.25*\unit) node {
    $\begin{array}{c}
      a+b\le10\\
      b+c\le11
    \end{array}$
    };
  \end{tikzgraph}
}
\hspace*{\fill}
\\
\vspace{0.25cm}
\hspace*{\fill}
\subfloat[D1]{%
  \begin{tikzgraph}
    \useasboundingbox (-1.95*\unit,-2.9*\unit) rectangle (1.7*\unit,1.6*\unit);
    \foreach \flip/\angle in {1/30,1/180,-1/30}
      \draw[yscale=\flip,rotate=\angle,fill=gray!50] \smallface;
    \draw (30:1.732*\unit) node[vertex] (a0) {$c$};
    \draw (0,0) node[vertex] (a1) {$b$};
    \draw (-30:1.732*\unit) node[vertex] (a2) {};
    \draw (-1.732*\unit,0) node[vertex] (a3) {$a$};
    \draw (a0) edge (a1) (a1) edge (a2) (a2) edge (a0) (a1) edge (a3);
    \draw (-0.125*\unit,-2.25*\unit) node {
    $\begin{array}{c}
      a+b\le10\\
      b+c\le11
    \end{array}$
    };
  \end{tikzgraph}
}
\hspace*{\fill}
\subfloat[D2]{%
  \begin{tikzgraph}
    \useasboundingbox (-1.95*\unit,-2.9*\unit) rectangle (1.7*\unit,1.6*\unit);
    \foreach \flip/\angle in {1/30,1/180,-1/30,-1/180}
      \draw[yscale=\flip,rotate=\angle,fill=gray!50] \smallface;
    \draw (30:1.732*\unit) node[vertex] (a0) {$c$};
    \draw (0,0) node[vertex] (a1) {$b$};
    \draw (-30:1.732*\unit) node[vertex] (a2) {};
    \draw (-1.732*\unit,0) node[vertex] (a3) {$a$};
    \draw (a0) edge (a1) (a1) edge (a2) (a2) edge (a0) (a1) edge (a3);
    \draw (-0.125*\unit,-2.25*\unit) node {
    $\begin{array}{c} 
      a+b\le11\\
      b+c\le11
    \end{array}$
    };
  \end{tikzgraph}
}
\hspace*{\fill}
\subfloat[D3]{%
  \begin{tikzgraph}
    \useasboundingbox (-1.95*\unit,-2.9*\unit) rectangle (1.7*\unit,1.6*\unit);
    \draw[yscale=-1,rotate=30,fill=gray!50] \smallface;
    \draw (30:1.732*\unit) node[3-vertex] (a0) {};
    \draw (0,0) node[vertex] (a1) {$7$};
    \draw (-30:1.732*\unit) node[vertex] (a2) {};
    \draw (-1.732*\unit,0) node[2-vertex] (a3) {};
    \draw (a0) edge (a1) (a1) edge (a2) (a2) edge (a0) (a1) edge (a3);
  \end{tikzgraph}
}
\hspace*{\fill}
\subfloat[D4]{%
  \begin{tikzgraph}
    \useasboundingbox (-1.95*\unit,-2.9*\unit) rectangle (1.7*\unit,1.6*\unit);
    \foreach \flip/\angle in {1/30,-1/30}
      \draw[yscale=\flip,rotate=\angle,fill=gray!50] \smallface;
    \draw (30:1.732*\unit) node[4-vertex] (a0) {};
    \draw (0,0) node[vertex] (a1) {$7$};
    \draw (-30:1.732*\unit) node[vertex] (a2) {};
    \draw (-1.732*\unit,0) node[2-vertex] (a3) {};
    \draw (a0) edge (a1) (a1) edge (a2) (a2) edge (a0) (a1) edge (a3);
  \end{tikzgraph}
}
\hspace*{\fill}
\\
\vspace{0.25cm}
\hspace*{\fill}
\subfloat[E1]{%
  \begin{tikzgraph}
    \useasboundingbox (-1.6*\unit,-2.7*\unit) rectangle (2.1*\unit,2.1*\unit);
    \draw[yscale=-1,rotate=45,fill=gray!50] \smallface;
    \draw (0,0) node[vertex] (a0) {};
    \draw (135:1.732*\unit) node[vertex] (a1) {};
    \draw (90:1.732*\unit) node[3-vertex] (a2) {};
    \draw (45:1.732*\unit) node[vertex] (a3) {};
    \draw (0:1.732*\unit) node[4-vertex] (a4) {};
    \draw (-45:1.732*\unit) node[vertex] (a5) {};
    \foreach \i in {1,...,5}
      \draw (a0) edge (a\i);
    \foreach \i/\j in {1/2,2/3,3/4,4/5}
      \draw (a\i) edge (a\j);
  \end{tikzgraph}
}
\hspace*{\fill}
\subfloat[E2]{%
  \begin{tikzgraph}
    \useasboundingbox (-1.6*\unit,-2.7*\unit) rectangle (2.1*\unit,2.1*\unit);
    \draw (0,0) node[vertex] (a0) {};
    \draw (135:1.732*\unit) node[vertex] (a1) {};
    \draw (90:1.732*\unit) node[3-vertex] (a2) {};
    \draw (45:1.732*\unit) node[vertex] (a3) {};
    \draw (0:1.732*\unit) node[4-vertex] (a4) {};
    \draw (-45:1.732*\unit) node[vertex] (a5) {$6$};
    \foreach \i in {1,...,5}
      \draw (a0) edge (a\i);
    \foreach \i/\j in {1/2,2/3,3/4,4/5}
      \draw (a\i) edge (a\j);
  \end{tikzgraph}
}
\hspace*{\fill}
\subfloat[E3]{%
  \begin{tikzgraph}
    \useasboundingbox (-2.1*\unit,-2.9*\unit) rectangle (2.1*\unit,2.1*\unit);
    \draw (0,0) node[vertex] (a0) {};
    \draw (180:1.732*\unit) node[vertex] (a1) {$a$};
    \draw (225:1.732*\unit) node[vertex] (a2) {$b$};
    \draw (270:1.732*\unit) node[vertex] (a3) {$c$};
    \draw (315:1.732*\unit) node[vertex] (a4) {$d$};
    \draw (0:1.732*\unit) node[vertex] (a5) {$e$};
    \draw (45:1.732*\unit) node[vertex] (a6) {};
    \draw (90:1.732*\unit) node[3-vertex] (a7) {};
    \draw (135:1.732*\unit) node[vertex] (a8) {};
    \foreach \i in {1,...,8}
      \draw (a0) edge (a\i);
    \foreach \i/\j in {1/2,2/3,3/4,4/5,5/6,6/7,7/8,8/1}
      \draw (a\i) edge (a\j);
    \draw (0,-2.5*\unit) node {$4 \in \{a,b,c,d,e\}$};
  \end{tikzgraph}
}
\hspace*{\fill}
\subfloat[E4]{%
  \begin{tikzgraph}
    \useasboundingbox (-2.1*\unit,-2.9*\unit) rectangle (2.1*\unit,2.1*\unit);
    \draw[yscale=-1,rotate=45,fill=gray!50] \smallface;
    \draw (0,0) node[vertex] (a0) {$6$};
    \draw (90:1.732*\unit) node[3-vertex] (a3) {};
    \draw (30:1.732*\unit) node[2-vertex] (a2) {};
    \draw (-45:1.732*\unit) node[vertex] (a1) {$a$};
    \foreach \i in {1,2}
      \draw (a0) edge (a\i);
    \draw (a2) edge (a3);
    \draw (0,-2.5*\unit) node {$a\le5$};
  \end{tikzgraph}
}
\hspace*{\fill}
\caption{The reducible configurations.}
\label{fig:configs}
\end{figure}


\section{Discharging rules}
\label{sec:disdesc}

Recall that $G=(V,E,F)$ is a plane graph that is a minimum counter-example to the
statement of Theorem~\ref{thm:main}, in the sense that $\abs{E}$ is
minimum. (In particular, a planar embedding of $G$ is fixed.)
We obtain a contradiction by using the Discharging
Method.
Each vertex and face of $G$ is assigned an initial charge;
the total sum of the charge is negative by Euler's Formula.
Then vertices and faces send or receive charge according
to certain redistribution rules. The total sum of the charge remains
unchanged, but ultimately (by using all of the reducible configurations in Section~\ref{sec:reddesc}) we deduce
that the charge of each face and vertex is non-negative, a contradiction.

\subsection{Initial charge}
We assign a charge to each vertex and face.
For every vertex $v\in V$, we define the initial charge
$\ch(v)$ to be $2\cdot\deg(v) - 6$, while for every face $f\in F$,
we define the initial charge $\ch(f)$ to be $\deg(f)-6$.
The total sum is
\[ \sum_{v\in V} \ch(v) + \sum_{f\in F} \ch(f) = -12\,. \]
Indeed, by Euler's formula $\abs{E}-\abs{V}-\abs{F}=-2$. Thus, $6\abs{E}-6\abs{V}-6\abs{F}=-12$.
Since $\sum_{v\in V}\deg(v)=2\abs{E}=\sum_{f\in F}\deg(f)$, it follows that
\begin{align*}
-12&=4\cdot\abs{E}-6\cdot\abs{V}+\sum_{f\in F}(\deg(f)-6)\\
&=\sum_{v\in V}(2\deg(v)-6)+\sum_{f\in F}(\deg(f)-6)\,.
\end{align*}

\subsection{Rules}
We need the following definitions to state the discharging rules. Given an $(\sge7)$-vertex $v$,
a face is \emph{special} (for $v$) if it is
an $(\sge5)$-face that is incident to a degree $2$ neighbour of $v$ (and so, in
particular, such a face is incident to $v$). Given a $6$-vertex $v$, a face $f$ is \emph{exceptional} (for $v$)
if $f$ is a $6$-face $vv_1v_2\ldots v_5$ where $v_1$ is a $2$-vertex and $v_2$ is a $3$-vertex.

Since $G$ may have cut-vertices (of a type not forbidden by
Lemma~\ref{cutvertex}), some vertices may be incident to
the same face several times. Thus, in the rules below, when we say that
a vertex or a face sends charge to an incident face or vertex, we mean
that the charge is sent as many times as these elements are incident to
each other.

The following describe how the charge is redistributed among the edges and faces in $G$.

\begin{itemize}
\item[\textbf{R0}]  An $(\sge4)$-face sends $1$ to each incident $2$-vertex.

\item[\textbf{R1}]  An $(\sge7)$-vertex sends
\begin{itemize}
\item[\textbf{R1a}]  $3/2$ to incident $(3,\sge7,\sge7)$-triangles and $(4,6,\sge7)$-triangles;

\item[\textbf{R1b}]  $7/5$ to incident $(5,5,\sge 7)$-triangles;

\item[\textbf{R1c}]  $5/4$ to incident $(4,\sge7,\sge7)$-triangles and $(2,8,4,8)$-faces;

\item[\textbf{R1d}]  $6/5$ to incident $(5,6,8)$-triangles;

\item[\textbf{R1e}]  $11/10$ to incident $(5,6,7)$- and $(5,\sge7,\sge7)$-triangles, and incident $(2,8,5,8)$-faces;

\item[\textbf{R1f}]  $1$ to incident $(\sge6,\sge6,\sge6)$-triangles, to every incident $4$-face that is not a $(2,8,\sle5,8)$-face, and to each of its incident special faces;

\item[\textbf{R1g}]  $1/2$ to each of its incident non-special $5$-faces.
\end{itemize}

\item[\textbf{R2}]  A $6$-vertex sends
\begin{itemize}
\item[\textbf{R2a}]  $11/10$ to incident $(5,6,6)$- and $(5,6,7)$-triangles;

\item[\textbf{R2b}]  $1$ to every other incident triangle, to each incident $4$-face, and to each of its incident exceptional faces;

\item[\textbf{R2c}]  $1/2$ to each of its incident $5$-faces and to each of its incident unexceptional $6$-faces.
\end{itemize}

\item[\textbf{R3}]  A $5$-vertex sends $4/5$ to each incident face.

\item[\textbf{R4}]  A $4$-vertex sends $1/2$ to each incident face.
\end{itemize}

\noindent
\emph{Note on rules R1 and R2.}  Since configurations A1, B1 and B2 are reducible, it follows from rule R1 that an $(\sge7)$-vertex sends positive charge to every incident triangle.
We conclude that an $(\sge7)$-vertex sends zero charge only to incident
$(\sge6)$-faces that are not special.
Similarly, a $6$-vertex sends zero charge only to incident $(\sge7)$-faces.


\section{Proof of Theorem~\ref{thm:main}}
\label{sec:disproof}

In this section, we prove that the final charge $\ch^*(x)$ of every
$x\in V\cup F$ is non-negative. Hence, we obtain
\[
-12=\sum_{x\in V\cup F}\ch(x)=\sum_{x\in V\cup F}\ch^*(x)\ge0,
\]
a contradiction. This contradiction establishes Theorem~\ref{thm:main}.

\subsection{Final charge of faces}

Let $f$ be a $d$-face. Our goal is to show that $\ch^*(f)\ge0$. Recall that the initial charge of $f$ is $\ch(f)=d-6$.

First suppose that $d\ge6$.  Let $p$ be the number of occurrences of an $(\sge7)$-vertex having $f$ as an incident special face, and $q$ the number of $2$-vertices incident to
$f$. In particular, $\ch^*(f)\ge d-6-q+p$ by rules R0 and R1f.
We define $x$ to be the number of edges of $f$ between a $2$-vertex and
an $(\sle6)$-vertex, $y$ the number of edges of $f$ between a $2$-vertex and
an $(\sge7)$-vertex,
and $z$ the number of edges of $f$ between two $2$-vertices. We have $2q=x+y+z$ and $2p\ge y$.
If $x=0$, then $p\ge q$,
and hence $\ch^*(f)\ge0$. Assume now that $x\ge1$.
Then Lemma~\ref{lem-new} implies that $2d-q-x\ge9$, that is $d-x/2\ge(q+9)/2$. Now, $\ch^*(f)\ge d-6-q+p\ge \lceil d-6-(x+z)/2 \rceil$
since $p\ge y/2$, $q=(x+y+z)/2$ and $d-6-q+p$ is integral. Hence,
\[
\ch^*(f)\ge\left\lceil\frac{q-3-z}{2}\right\rceil,
\]
which is non-negative if $q-z\ge2$. It remains to deal with the case where $q-z\le 1$.  Note that $q\ge 2z$ due to Lemma~\ref{lem-newnew}.  It therefore follows that $z\le 1$.  Let us first consider the case $z=1$, hence $q=2$.  In this case, it follows by Lemma~\ref{lem-newnew} that $p\ge2$ and therefore $\ch^*(f)\ge d-6-2+2\ge0$.  We just need to check the case $z=0$ and $q=1$ (since $x\ge1$). Then we may assume that $d=6$, for if
$d\ge7$ then $\ch^*(f)\ge d-6-1\ge0$. Moreover, if $y\ge1$ then $\ch^*(f)\ge0$ by rule R1f. Hence, $x=2q-y-z=2$. Let $v$ and $v'$
be the two $(\sle6)$-neighbours of the $2$-vertex of $f$. First, if both $v$ and $v'$ have degree more than $3$, then each of them
sends at least $1/2$ to $f$ by rules R2, R3 and R4, and hence
$\ch^*(f)\ge0$. So, as $q=1$, we may assume that $v$ has degree $3$. Now, by the reducibility of
configuration A3, the degree of $v'$ is at least $6$ and hence exactly $6$.
Consequently, $f$ is exceptional for $v'$ and $f$ receives $1$ from $v'$ by rule R2b, which concludes the analysis for $(\sge6)$-faces.

Suppose that $d=5$, and let $q$ be the number of $2$-vertices incident to $f$.
Lemma~\ref{lem-new} implies that $f$ is incident to at most one loose edge.
Thus, if $q=0$ (so that $f$ sends no charge to vertices) then $f$ is incident to at least two $(\sge4)$-vertices,
and hence $\ch^*(f)\ge5-6+2\cdot1/2=0$ by rules R1g, R2c, R3 and R4.
Moreover, if $q\ge1$, then Lemma~\ref{lem-new} implies that $f$ is not incident to a loose edge and $f$ is thus incident to at least $q+1$ vertices of degree at least $7$.
herefore, $f$ receives at least $q+1$ from its incident $(\sge7)$-vertices by rule R1f, and so $\ch^*(f)\ge5-6-q+q+1=0$.

Next suppose that $d = 4$.  Let the four vertices incident to $f$ be
$v_0,\ldots,v_3$ in clockwise order and suppose without loss of generality that $v_0$ has the least degree among $v_0,\ldots,v_3$.
First, if $\deg(v_0)\ge4$, then by rules R1f, R2b, R3 and R4, the charge sent to $f$ by each incident vertex is at least $1/2$,
so that $\ch^*(f) \ge -2+4\cdot1/2=0$.
If $\deg(v_0) = 3$,
then since configuration B1 is reducible $\deg(v_1) \ge 7$ and $\deg(v_3)\ge
7$. Thus, by rule R1f, $\ch^*(f) \ge -2 + 2 = 0$.
Last, assume that $\deg(v_0) = 2$. Since configuration B1 is
reducible, $\deg(v_1)=\deg(v_3)=8$, and since configuration A2 is
reducible, $\deg(v_2) \ge 4$. By rule R0, $f$ sends charge $1$ to $v_0$.
But $f$ receives charge $3$: by rules R1c and R4 if $f$ is a $(2,8,4,8)$-face;
by rules R1e and R3 if $f$ is a $(2,8,5,8)$-face; and by
rules R1f and R2b if $f$ is a $(2,8,\sge6,8)$-face. Thus, $\ch^*(f)\ge0$.

Finally suppose that $d = 3$.  Let the three vertices incident to $f$ be $v_0$, $v_1$ and $v_2$, and let us assume without loss of generality that $\deg(v_0)\le\deg(v_1)\le\deg(v_2)$.
Since configuration A1 is reducible, $\deg(v_0)\ge3$.  Thus $f$ sends no
charge, but needs to make up for an initial charge of $-3$.  We analyse
several cases according to the value of $\deg(v_0)$.
\begin{description}
\item{$\deg(v_0) = 3$.}  Since configuration B1 is reducible, $\deg(v_1) \ge 7$.  By
rule R1a, $f$ receives charge $2\cdot3/2 =3$.

\item{$\deg(v_0) = 4$.}  Since configuration B1 is reducible, $\deg(v_1) \ge 6$.
If $\deg(v_1)\ge7$, then $f$ receives charge $2\cdot5/4+1/2=3$ by rules
R1c and R4. Otherwise, $\deg(v_1) = 6$ and hence $\deg(v_2) \ge 7$ since configuration B2 is reducible, but
then $f$ receives charge $3/2 + 1 + 1/2 = 3$ by rules R1a, R2b and R4.

\item{$\deg(v_0) = 5$.}  If $\deg(v_1) = 5$, then $\deg(v_2) \ge 7$ since configuration B2
is reducible, but then $f$ receives charge $7/5+2\cdot4/5=3$ by rules
R1b and R3.  If $\deg(v_1) = 6$, then we separately consider the cases of $\deg(v_2) \in
\{6,7,8\}$.  If $\deg(v_2) \in \{6,7\}$, then $f$ receives charge
$2\cdot11/10+4/5=3$ by rules R1e, R2a and R3; if $\deg(v_2) = 8$, then $f$
receives charge $6/5+1+4/5=3$ by rules R1d, R2b and R3. Last, if
$\deg(v_1)\ge7$, then $f$ receives charge $2\cdot11/10+4/5=3$ by rules R1e and
R3.

\item{$\deg(v_0) \ge 6$.} The face $f$ receives charge at least $3$ by rules R1f and R2b.
\end{description}
This concludes our analysis of the final charge of $f$, verifying that $\ch^*(f) \ge 0$.


\subsection{Final charge of $(\sle6)$-vertices}

Let $v$ be an arbitrary vertex of $G$. Our goal is to show that $\ch^*(v)\ge0$.  Recall that the initial charge  of $v$ is $\ch(v)=2\cdot\deg(v)-6$. Moreover, $\deg(v)\ge2$ since configuration A0 is reducible.

If $\deg(v)=2$, then $v$ is incident to two $(\sge4)$-faces since configuration A1 is reducible; thus, $v$ receives charge $1$ from both incident faces by rule R0 and the final charge of $v$ is $\ch^*(v) = -2 + 2 = 0$.

If $\deg(v)=3$, then $v$ neither sends nor receives any charge; hence, the final charge of $v$ is $\ch^*(v)=\ch(v)=0$.

If $\deg(v) \in \{4,5\}$, then $v$ sends charge $\ch(v)/\deg(v)$ to each incident face by rules R3 and R4; the final charge of $v$ is $\ch^*(v) =0$.

Suppose now that $\deg(v) = 6$.  The initial charge of $v$ is $\ch(v) =
6$. 
If $v$ is incident to a $5$-face or an unexceptional $(\sge6)$-face, then it sends charge at most $1/2$ to one of the faces by rule R2c and so by rule R2 the total charge sent by $v$ is at most $5\cdot 11/10 +1/2 = 6$.
If $v$ is incident to an exceptional $6$-face, then, since configuration E4 is reducible, $v$ has no incident $(5,6,6)$- or $(5,6,7)$-triangles and thus the total charge sent is at most $5\cdot1+1=6$ by rules R2b and R2c.
We conclude that $v$ is only incident to $(\sle4)$-faces. Then,
since configuration C2 is reducible, $v$ has no incident
$(5,6,6)$-face; furthermore, since configuration C3 is reducible,
$v$ has no incident $(5,6,7)$-face. Therefore, the charge sent by $v$
is at most $6$ and the final charge of $v$ satisfies $\ch^*(v) \ge 0$.

\subsection{Final charge of $7$-vertices}

Next, suppose that $\deg(v) = 7$.
For convenience, let $v_0,v_1,\ldots,v_6$ be the neighbours of $v$ in clockwise order, and let $f_i$ be the face $vv_iv_{i+1}$ for $i\in\{0,1,\ldots,6\}$, where the index is modulo $7$.
The initial charge of $v$ is $\ch(v) = 8$.
We partition our analysis based on the number of incident special $(\sge5)$-faces.  Note that since configuration B1 is reducible, if $v$ is adjacent to a $2$-vertex then both of the $2$-vertex's incident faces are special for $v$.

\subsubsection{There is an adjacent $2$-vertex}\label{subsub:2vertex}

We first treat the cases in which $v$ is adjacent to some $2$-vertex. 
In these cases, there are at least two incident special $(\sge5)$-faces.  Thus, we may assume that $v$ is incident to at most one non-special $(\sge5)$-face (which is sent at most $1/2$ charge by rule R1g), for otherwise the total charge sent by $v$ is at most $3\cdot3/2+2+2\cdot1/2<8$.
Now note that, by rules R1a and R1b, any face that is sent charge more than $5/4$ must be a $(3,7,\sge7)$-, $(4,6,7)$- or $(5,5,7)$-triangle.
And so we assert that if $f_i$ is such a triangle, then both $f_{i-1}$ and $f_{i+1}$ are $(\sge5)$-faces.
The assertion holds if $f_i$ is a $(3,7,\sge7)$-triangle since configurations C1 and D3 are reducible, and the fact that configuration B3 is reducible implies the assertion for the two other cases.

\paragraph{Case \ref{subsub:2vertex}(1).}
If $v$ is incident to (exactly) one non-special $(\sge5)$-face, then $v$ is incident to only two special $(\sge5)$-faces (for otherwise the total charge sent by $v$ is $3\cdot3/2+3+1/2=8$). Thus, the remaining four incident faces are $(\sle4)$-faces. 
Observe that there are at least three of these faces that are adjacent
around $v$ to another $(\sle4)$-face.
Hence
by the assertion at the end of the last paragraph, each of these three
faces is sent at most $5/4$ charge.
So the charge sent by $v$ is at most $3/2+3\cdot5/4+2+1/2<8$.
Thus, $v$ is not incident to a non-special $(\sge5)$-face.

\paragraph{Case \ref{subsub:2vertex}(2).}
If $v$ is incident to at least five $(\sge5)$-faces, then the
charge sent is at most $2\cdot 3/2+5 = 8$ due to rule R1f.

\paragraph{Case \ref{subsub:2vertex}(3).}
If $v$ is incident to exactly four $(\sge5)$-faces, all of which are
special, then there must be two incident $(\sle4)$-faces that are
adjacent. (Recall that each special face is adjacent to another special
face.) By the assertion in the second paragraph of the $7$-vertex
analysis, both of these are sent charge at most $5/4$.  Therefore, the
total charge sent by $v$ in this case is at most $3/2+2\cdot5/4+4=8$.

\paragraph{Case \ref{subsub:2vertex}(4).}
If $v$ is incident to exactly three $(\sge5)$-faces, all of which are special, then these faces are sequentially adjacent around $v$.  Hence, by the assertion in the second paragraph of the $7$-vertex analysis, no face is sent charge more than $5/4$ and the total charge sent is at most $4\cdot5/4+3=8$

\paragraph{Case \ref{subsub:2vertex}(5).}
Suppose that $v$ is incident to exactly two
$(\sge5)$-faces, say $f_0$ and $f_1$, both special (so $v_1$ is a $2$-vertex).
Recall that all other incident faces have size at most $4$.
Let us analyse which incident faces can be sent
charge $5/4$. By rule R1c, such a face must be a $(4,7,\sge7)$-triangle.
Since configuration D4 is reducible, such a face must be adjacent to
a special face for $v$. Thus, there are at most two such
faces, namely $f_2$ and $f_6$. Consequently,
the total charge sent by $v$ is at most
$2\cdot5/4+3\cdot11/10+2<8$ by rules R1c, R1e and R1f.

\subsubsection{There is no adjacent $2$-vertex}\label{subsub:no2vertex}

Now we may assume that $v$ is not adjacent to a $2$-vertex.
In these cases, we may assume that $v$ is incident to at most two (non-special) $(\sge5)$-faces (which are sent at most $1/2$ charge by rule R1g), for otherwise the total charge sent by $v$ is at most $4\cdot3/2+3\cdot1/2<8$.

\paragraph{Case \ref{subsub:no2vertex}(1).}
Suppose that $v$ is incident to two $5$-faces (and hence to no other $(\sge5)$-faces).  First suppose that there are four $(\sle4)$-faces that are sequentially adjacent around $v$.  Since configuration B3 is reducible, none of these is a $(3,7,7)$-, $(4,6,7)$- or $(5,5,7)$-triangle. Also, since configuration C1 is reducible, at most two of these are $(3,7,8)$-triangles.  It follows that at most three faces are sent charge $3/2$, and the remaining two $(\sle4)$-faces are sent charge at most $5/4$; so the total charge sent by $v$ is at most $3\cdot3/2+2\cdot5/4+2\cdot1/2=8$. Next suppose that there are only three $(\sle4)$-faces that are sequentially adjacent around $v$. As before, we deduce that none of these is a $(3,7,7)$-, $(4,6,7)$- or $(5,5,7)$-triangle and at most two of these are $(3,7,8)$-triangles. If two of these faces are $(3,7,8)$-triangles, then the middle one is either a $(7,8,8)$-triangle or a $4$-face, and hence sent charge $1$ by rule R1f. Hence, the total charge sent by $v$ is at most $\max\{4\cdot3/2+1+2\cdot1/2,3\cdot3/2+2\cdot5/4+2\cdot1/2\}=8$.

\paragraph{Case \ref{subsub:no2vertex}(2).}
Suppose that $v$ is incident to exactly one $5$-face, say it is $f_0$ without loss of generality. Then $v$ cannot be adjacent to an
$(\sge6)$-face, for otherwise the total charge sent by $v$ is at most
$5\cdot3/2+1/2=8$.  As above, since configurations B3 and C1 are
reducible, none of the remaining faces (all $(\sle4)$-faces) is a
$(3,7,7)$-, $(4,6,7)$- or $(5,5,7)$-triangle and at most two of them are
$(3,7,8)$-triangles (either $f_1$ or $f_6$).  Indeed, one of $f_1$ and
$f_6$, must be a $(3,7,8)$-triangle, for otherwise the charge sent by
$v$ is at most $6\cdot5/4+1/2=8$. Assume without loss of generality that
$v_0$ has degree $3$. Therefore, since configuration D1 is
reducible, for each $i\in\{2,3,4,5\}$, either $f_i$ is a $4$-face or both $v_i$ and $v_{i+1}$ have degree at least $5$.
It follows that each of $f_2$, $f_3$, $f_4$ and $f_5$ is sent charge at most $11/10$;
thus, $v$ sends total charge at most $2\cdot3/2+4\cdot11/10+1/2<8$.

\paragraph{Case \ref{subsub:no2vertex}(3).}
It cannot be that $v$ is incident to two $(\sge6)$-faces, for then the total charge sent by $v$ would be at most $5\cdot3/2 < 8$. The case in which $v$ is incident to exactly one $(\sge6)$-face is handled by an argument identical to the one used in the previous paragraph.

\paragraph{Case \ref{subsub:no2vertex}(4).}
If $v$ is incident only to $(\sle4)$-faces, then, since configurations B3 and C1 are reducible, $v$ is not incident to a $(3,7,\sge7)$-, $(4,6,7)$- or $(5,5,7)$-triangle.
If $v$ is incident to a $(4,7,\sge7)$-triangle, then, since configuration D2 is reducible, $v$ cannot be adjacent to any other $4$-vertex.
It therefore follows that the total charge sent by $v$ in this case is at most $2\cdot 5/4+5\cdot11/10=8$.

\medskip
This concludes the analysis of the final charge of the $7$-vertices.

\subsection{Final charge of $8$-vertices}

Last, suppose that $\deg(v) = 8$.
For convenience, let $v_0,\ldots,v_7$ be the neighbours of $v$ in clockwise order, and for $i\in\{0,\ldots,7\}$, let $f_i$ be the face of $G$ incident with $vv_i$ and $vv_{i+1}$, where the index is modulo $8$.
The initial charge of $v$ is $\ch(v)
= 10$. We partition our analysis based on the number of incident
special faces. Note that since configuration C1 is
reducible, if $v$ is adjacent to a $2$-vertex then at least one of the
$2$-vertex's incident faces is special for $v$.
Furthermore, since configurations A2 and B1 are reducible, if one of
the $2$-vertex's incident faces is an $(\sle4)$-face, then it must be a
$(2,8,\sge4,8)$-face.

\subsubsection{There is an incident special face}\label{subsub:special}

We start with the cases in which there is a face that is special for $v$.
In these cases, we may assume that there is at most one incident non-special $(\sge5)$-face (that is sent charge at most $1/2$ by rule R1g), for otherwise the total charge sent by $v$ is at most $5\cdot3/2+1+2\cdot1/2<10$.

Suppose that $v$ is incident to (exactly) one non-special
$(\sge5)$-face. It may not be that $v$ is incident to more than one
special face, since then $v$ would send charge at most
$5\cdot3/2+2+1/2=10$. So the remaining faces are $(\sle4)$-faces. Let us
suppose that $v_0$ is the $2$-vertex, $f_0$ is the special face and
$f_7$ is a $(2,8,\sge4,8)$-face (sent charge at most $5/4$ by rules R1c, R1e and R1f), without
loss of generality. Among the faces $f_1,\ldots,f_6$, (disregarding
which one is the non-special $(\sge5)$-face), there must be three $(\sle4)$-faces that are sequentially adjacent around $v$. Since configurations B4 and D1 are reducible, the middle of these faces may not be a $(3,\sge7,8)$-, $(4,6,8)$- or $(5,5,8)$-triangle and hence is sent charge at most $5/4$. Therefore, the total charge sent by $v$ is at most $4\cdot3/2+2\cdot5/4+1+1/2=10$.

So we assume now that every $(\sge5)$-face incident to $v$ is special
for $v$.
If $v$ has at least four incident special faces, then the
charge sent is at most $4\cdot 3/2+4 = 10$.

\paragraph{Case \ref{subsub:special}(1).}
Suppose that $v$ is incident to exactly three special faces.
If $v$ is incident to at least two $(2,8,\sge4,8)$-faces, each sent charge at most $5/4$ by rule R1c, then the total charge sent by $v$ is at most $3\cdot3/2+2\cdot5/4+3=10$.
If $v$ is incident to exactly one $(2,8,\sge4,8)$-face, then it must be that $v$ is incident to three sequentially adjacent $(\sle4)$-faces, say $f_0$, $f_1$ and $f_2$.  Since configuration B4 is reducible, $f_1$ is not a $(3,7,8)$-, $(4,6,8)$- or $(5,5,8)$-face; since $v$ is incident to a $(2,8,\sge4,8)$-face and configuration D1 is reducible, $f_1$ is not a $(3,8,8)$-face; hence, $f_1$ receives charge at most $5/4$. Consequently, the total charge sent by $v$ is at most $3\cdot 3/2+2\cdot5/4+3=10$.
If $v$ is not incident to a $(2,8,\sge4,8)$-face, then the three incident special faces are sequentially adjacent around $v$.
In the following we shall assume that $f_5$, $f_6$ and $f_7$ are the three special faces.
The remaining five faces are all triangles, otherwise (by rule R1f) the total charge sent by $v$ is at most $4\cdot3/2+4=10$.
Note that there is no $i \in \{1,2,3\}$ such that all of $f_{i-1},f_i,f_{i+1}$ are $(3,\sge7,8)$-triangles since configuration C4 is reducible.  Since configuration B4 is reducible, none of $f_1$, $f_2$, $f_3$ is a $(3,7,8)$-, $(4,6,8)$- or $(5,5,8)$-triangle.  Furthermore, since configuration D2 is reducible, $v$ is incident to at most one pair of adjacent $(3,\sge7,8)$-triangles.  Thus, at most one of $v_1,\ldots,v_4$ is a $3$-vertex.
\begin{description}
\item If none of $v_1,\ldots,v_4$ is a $3$-vertex, then the only faces that can be sent charge more than $5/4$ are $f_0$ and $f_4$.  Therefore, the total charge sent by $v$ is at most $2\cdot3/2+3\cdot5/4+3<10$.
\item Suppose that $v_2$ or $v_3$ is a $3$-vertex, say $v_2$ by
symmetry.  Then $v_1$ and $v_3$ are $8$-vertices. Since 
configuration C4 is reducible, $v_0$ and $v_4$ are $(\sge4)$-vertices and hence $f_0$ and $f_3$ are each sent charge at most $5/4$.  Thus, the total charge sent by $v$ is at most $3\cdot3/2+2\cdot5/4+3=10$.
\item Suppose that $v_1$ or $v_4$ is a $3$-vertex, say $v_1$ by
symmetry.  Then $f_4$ is the only face other than $f_0$ and $f_1$ that
can be sent charge more than $5/4$. In this case, the total charge sent by $v$ is at most $3\cdot3/2+2\cdot5/4+3=10$.
\end{description}

\paragraph{Case \ref{subsub:special}(2).}
Suppose that $v$ is incident to exactly two special faces (and hence is incident to at most two $(2,8,\sge4,8)$-faces).  First, assume that $v$ is incident to a $(2,8,\sge4,8)$-face.
Since configuration B4 is reducible, if $f_i$ is a $(3,7,8)$-,
$(4,6,8)$- or $(5,5,8)$-triangle, then $f_{i-1}$ or $f_{i+1}$ is an $(\sge5)$-face; also, since configuration D1 is reducible (and $v$ is incident to a $(2,8,\sge4,8)$-face), the same conclusion holds if $f_i$ is a $(3,8,8)$-triangle.
Since each incident special face is sequentially adjacent either to an incident $(2,8,\sge4,8)$-face or to the other special face, we deduce that at most two faces are sent charge more than $5/4$.  Thus, the total charge sent by $v$ is at most $2\cdot3/2+4\cdot5/4+2=10$.

Now we deal with the case where $v$ is not incident to a
$(2,8,\sge4,8)$-face. Suppose $f_6$ and $f_7$ are the two special faces, with $v_7$ being a $2$-vertex.
Recall that none of $f_0,\ldots,f_5$ is an $(\sge5)$-face.
Note also that at most one of $f_0,\ldots,f_5$ is a $4$-face,
for otherwise the charge sent by $v$ is at most $4\cdot3/2+4=10$.
We analyse possible $(3,\sge7,8)$-triangles among these six faces.
First, note that there is no $i\in\{1,2,3,4\}$ such
that all of $f_{i-1},f_{i},f_{i+1}$ are $(3,\sge7,8)$-triangles since
configuration C4 is reducible.
Since configuration B4 is reducible, none of
$f_1,\ldots,f_4$ is a $(3,7,8)$-, $(4,6,8)$- or $(5,5,8)$-triangle.
Furthermore, since configuration D2 is reducible, $v$ is
incident to at most one pair of adjacent $(3,\sge7,8)$-triangles.

\paragraph{Case \ref{subsub:special}(2)(a).}
First, we suppose that all of $f_0,\ldots,f_5$ are triangles.  It follows that at most one of $v_1,\ldots,v_5$ is a $3$-vertex.
We consider several cases regarding which neighbours of $v$ are
$(\sle4)$-vertices. 
\begin{description}
\item If none
of $v_1,\ldots,v_5$ is a $3$-vertex, then the only faces that can be
sent charge more than $5/4$ are $f_0$ and $f_5$. Therefore, the total
charge sent by $v$ is at most $2\cdot3/2+4\cdot5/4+2=10$.
\item Suppose that $v_3$ is a $3$-vertex. Hence, $v_2$ and $v_4$ are
$8$-vertices. We show that $f_0$ and $f_1$ are sent charge at most
$5/2$ altogether by $v$. Indeed, if $v_1$ is a $4$-vertex, then $\deg(v_0)\ge7$ because 
configurations B1 and E2 are reducible. Hence, $v$ sends charge $5/4$ to each of
$f_0$ and $f_1$ by rule R1c. If $v_1$ is a $5$-vertex, then
$\deg(v_0)\ge5$ as configuration B1 is reducible, and hence $v$ sends charge $11/10$ to $f_1$ and at most
$7/5$ to $f_0$ by rules R1b and R1e. Last, if $\deg(v_1)\ge6$ then $v$
sends charge $1$ to $f_1$ and at most $3/2$ to $f_0$ by rules R1a and
R1f. Similarly, we deduce that $v$ sends charge at most $5/2$ to
$f_4$ and $f_5$ altogether. Therefore, the total charge sent by $v$ is
at most $2\cdot3/2+2\cdot5/2+2=10$.
\item Suppose that $v_2$ or $v_4$ is a $3$-vertex, say $v_2$ by
symmetry. Then, $v_1$ and $v_3$ are $8$-vertices. 
We have $\deg(v_0)\ge4$ since configuration C4 is reducible, so that $f_0$ receives charge at
most $5/4$ from $v$. Thus, it suffices to show that $v$ sends to $f_3$, $f_4$
and $f_5$ charge at most $15/4$ altogether: the total charge sent by
$v$ would then be at most $2\cdot3/2+5/4+15/4+2=10$.
First, $\deg(v_4)\ge5$ since 
configuration E1 is reducible. Recall that $\deg(v_4)+\deg(v_5)\ge11$.
If $\deg(v_4)+\deg(v_5)\ge12$, then $f_3$ and $f_4$
are sent charge at most $9/4$ altogether by $v$. Thus, the conclusion
holds since $f_6$ is sent charge at most $3/2$ by $v$. Now, if
$\deg(v_4)+\deg(v_5)=11$, then
$\deg(v_5)+\deg(v_6)\ge11$ since configuration C5 is
reducible. Consequently, each of $f_4$ and $f_5$ is sent charge at most
$5/4$ by $v$: recall that none of $v_3$, $v_4$ and $v_5$ is a
$3$-vertex, and if $v_6$ were a $3$-vertex then $v_5$ would be an
$8$-vertex so that $v_4$ would have to be a $3$-vertex in order for the charge sent to be more
than $5/4$.
Moreover, $f_3$ is sent charge at most $11/10$ by rule
R1e, so that the conclusion holds.
\item Suppose that $v_1$ or $v_5$ is a $3$-vertex, say $v_1$ by symmetry. Then, $\deg(v_0)=8=\deg(v_2)$,
and $\deg(v_3)\ge5$ since configuration E1 is reducible. Further, recall that
$v_4$ and $v_5$ both have degree at least $4$.
If $\deg(v_6)\le4$, then $\deg(v_5)\ge6$. Since
$\deg(v_3)+\deg(v_4)\ge11$ (because configuration B4 is reducible), at least one of $f_2$ and $f_4$ is sent charge at most $1$, implying that the total charge sent by $v$ is at most $3\cdot3/2+2\cdot5/4+3=10$.
If $\deg(v_6)\ge5$, then $f_5$ is sent charge at most $7/5$ and $f_2$ is sent charge at most $11/10$, so the total charge sent by $v$ is at most $2\cdot3/2+7/5+2\cdot5/4+11/10+2=10$.
\end{description}

\paragraph{Case \ref{subsub:special}(2)(b).}
Assume now that (exactly) one of $f_0,\ldots,f_5$ is a $4$-face.  (Such a $4$-face is assumed to not have an incident $2$-vertex.)
Without loss of generality, we may suppose that it is one of $f_0$, $f_1$ and $f_2$.
Recall that none of $f_1,\ldots,f_4$ is a $(3,7,8)$-, $(4,6,8)$- or $(5,5,8)$-triangle since configuration B4 is reducible.
Also, at most one of $v_1,\ldots,v_5$ is a $3$-vertex since 
configurations C4 and D2 are reducible.
Moreover, if at most one of $f_1,\ldots,f_4$ is sent charge $3/2$ by $v$ (i.e.~is a $(3,8,8)$-triangle), then the total charge sent by $v$ is at most $3\cdot3/2+2\cdot5/4+3=10$.
In particular, we assume that (exactly) one of $v_2,v_3,v_4$ has degree $3$.
\begin{description}
\item Suppose that the $4$-face is $f_0$. At most three of
$f_1,\ldots,f_5$ are sent charge more than $5/4$, so
the total charge sent by $v$ is at most $3\cdot3/2+2\cdot5/4+3=10$.
\item Suppose that the $4$-face is $f_1$. By the remark above, one of
$v_3$ and $v_4$ has degree $3$.
If $\deg(v_4)=3$, then $\deg(v_3)=\deg(v_5)=8$.
Further, $\deg(v_6)\ge4$ since configuration C4 is reducible. Consequently,
the total charge sent by $v$ is at most $3\cdot3/2+2\cdot5/4+3=10$.
If $\deg(v_3)=3$, then $\deg(v_4)=8$ and we shall see that $f_4$ and $f_5$ are sent at most $5/2$ altogether by $v$. Indeed, let us check all of the subcases:
if $\deg(v_6)\le4$, then $\deg(v_5)\ge6$, implying that $f_4$ is sent charge $1$ and $f_5$ is sent charge at most $3/2$;
if $\deg(v_6)=5$, then $\deg(v_5)\ge5$, implying that $f_4$ is sent charge at most $11/10$ and $f_5$ is sent charge at most $7/5$;
if $\deg(v_6)=6$, then $\deg(v_5)\ge5$ since configuration E2 is reducible, and so each of $f_4$ and $f_5$ are sent charge at most $6/5$; if $\deg(v_6) \ge 7$, then each of $f_4$ and $f_5$ are sent charge at most $5/4$.
Therefore, the total charge sent by $v$ is at most $3\cdot3/2+3+5/2=10$.
\item Suppose that the $4$-face is $f_2$. Then, $\deg(v_4)=3$, for
otherwise at most one face among $f_1,\ldots,f_4$ is sent charge $3/2$
by $v$. Then, $\deg(v_5)=8$ and $\deg(v_6)\ge4$ since configuration C4 is reducible.
Therefore, the total charge sent by $v$ is at most $3\cdot3/2+2\cdot5/4+3=10$.
\end{description}

\paragraph{Case \ref{subsub:special}(3).}
Suppose that $v$ is incident to exactly one special face. Then $v$ is
incident to a $(2,8,\sge4,8)$-face and, since configurations B4 and
D1 are reducible, $v$ is incident to at most one $(3,\sge7,8)$-,
$(4,6,8)$- or $(5,5,8)$-face; the total charge sent by $v$ is at most
$3/2+6\cdot5/4+1=10$.

\subsubsection{There is no incident special face}\label{subsub:nospecial}

Suppose that $v$ is not incident to a special face.  Any $(\sge5)$-face
incident to $v$ is sent charge at most $1/2$ by rule R1g. If there are two such
faces, then the total charge sent by $v$ is at most
$6\cdot3/2+2\cdot1/2=10$.

\paragraph{Case \ref{subsub:nospecial}(1).}
Suppose there is exactly one incident
$(\sge5)$-face, say $f_0$.  None of the faces $f_2,\ldots,f_6$ is a
$(3,7,8)$-, $(4,6,8)$- or $(5,5,8)$-triangle since configuration B4
is reducible.  Since configuration D2 is reducible, among the vertices $v_2,\ldots,v_7$ there is at most one $3$-vertex that is incident to some triangle among $f_2,\dots,f_6$.  None of the vertices $v_2,\ldots,v_7$ is a $2$-vertex
since configuration C1 is reducible (so in particular none of
$f_2,\ldots,f_6$ is a $(2,8,\sle5,8)$-face). Consequently,
$f_1,\ldots,f_7$ are all triangles, or else the total charge sent
by $v$ is at most $4\cdot3/2+2\cdot5/4+1+1/2=10$.
Moreover, one of $v_2,\ldots,v_7$ must be a $3$-vertex
or else the total charge sent by $v$ is at most
$3\cdot3/2+4\cdot5/4+1/2=10$. Suppose,
without loss of generality, that
the $3$-vertex is $v_2$, $v_3$ or $v_4$.
If it is $v_2$, then at most three faces are sent charge $3/2$ and the
total charge sent by $v$ is at most $3\cdot3/2+4\cdot5/4+1/2=10$.  If it
is $v_3$, then $v_2$ has degree $8$ and $v$ sends charge at most $10$
unless $v_1$ has degree $3$; however, this contradicts the reducibility of
configuration C4. If it is $v_4$, then $v_3$ has degree
$8$ and there are three sub-cases. First, if $v_2$ has degree $6$, then
$v$ sends $f_2$ charge $1$ and charge at most
$4\cdot3/2+2\cdot5/4+1+1/2=10$ in total; second, if $v_2$ has degree
$5$, then $v$ sends $f_2$ charge at most $11/10$, $f_1$ charge at most
$7/5$, and charge at most $3\cdot3/2+7/5+11/10+2\cdot5/4+1/2=10$ in
total; third, if $v_2$ has degree $4$, then, since configuration E2
is reducible, $v_1$ has degree at least $7$, so $v$ sends $f_1$ and
$f_2$ each charge at most $5/4$, and charge at most
$3\cdot3/2+4\cdot5/4+1/2=10$ in total.

\paragraph{Case \ref{subsub:nospecial}(2).}
Finally, we are in the case that $v$ is only incident to $(\sle4)$-faces
(in particular none of which are $(2,8,\sle5,8)$-faces since configuration C1 is
reducible). Since configuration B4 is reducible, $v$ is incident to
no $(3,7,8)$-, $(4,6,8)$- or $(5,5,8)$-triangle.  If $v$ is not
incident to a $(3,8,8)$-triangle, then no face is sent charge more than
$5/4$ and hence the total charge sent by $v$ is at most $8\cdot5/4=10$.
So assume that $v_7$ is a $3$-vertex, and that $f_7$ is a
$(3,8,8)$-triangle. Since configuration D2 is reducible,
$v$ is adjacent to no other $3$-vertices. We may assume that $v$
is incident to fewer $4$-faces than the number of $(3,8,8)$-triangles
incident to $v$.  (Otherwise, if $x$ is the number of
$4$-faces incident to $v$, then the total charge sent by $v$
is at most $x\cdot(3/2+1)+(8-2\cdot x)\cdot5/4 = 10$.)
If $f_7$ were the only $(3,8,8)$-triangle then $f_6$ would
necessarily be a $4$-face.
We conclude, therefore, that $v$ is
incident to exactly two $(3,8,8)$-triangles, namely $f_6$ and $f_7$, and to at most one $4$-face. Now, if $v$ is incident to only triangles,
then since configuration E3 is reducible, every neighbour of $v$
other than $v_7$ has degree at least $5$, and so the total charge sent is
at most $2\cdot3/2+4\cdot6/5+2\cdot11/10 = 10$ (where we observe that
the faces adjacent around $v$ to the $(3,8,8)$-triangles are
$(\sge5,8,8)$-faces and hence sent charge at most $11/10$).

Therefore, in addition to the two $(3,8,8)$-triangles, $v$ must be
incident to exactly one $4$-face. By symmetry, let us assume that $f_0$, $f_1$ and $f_2$ are triangles. Since configuration E1 is reducible, $v_1$ has degree at least $5$.
If $v_1$ has degree at least $6$, then $v$ sends $f_0$ charge $1$ and total charge at most $2\cdot3/2+4\cdot5/4+2=10$.
If $v_1$ has degree $5$, then $v$ sends $f_0$ charge $11/10$ and $v_2$ has degree at least $6$. If $v_2$ has degree at least $7$, then $v$ sends $f_1$ charge $11/10$ and total charge at most $2\cdot3/2+3\cdot5/4+2\cdot11/10+1<10$. If $v_2$ has degree $6$, then $v$ sends $f_1$ and $f_2$ each charge at most $6/5$, and total charge at most $2\cdot3/2+2\cdot5/4+2\cdot6/5+11/10+1=10$.

\medskip
We have shown that if $\deg(v) = 8$, then $\ch^*(v)\ge 0$. This
allows us to conclude our analysis of the final charge of $v$, having shown $\ch^*(v) \ge 0$ in all cases.  This completes the proof of Theorem~\ref{thm:main}.\hfill$\Box$


\section{Proofs of reducibility}
\label{sec:redproof}

In this section, we prove that the graph $G$ cannot contain any of the configurations given in Section~\ref{sec:reddesc}.

Let $\lambda$ be a (partial) \col{9} of $G$.
For each element $x\in E\cup F$, we define ${\mathcal C}(x)$ to be the set of colours
(with respect to $\lambda$) of the edges and faces incident or adjacent to $x$.
If $x\in V$ we define ${\mathcal E}(x)$ to be the set of colours of the edges incident to $x$.
Moreover, $\lambda$ is \emph{nice} if only some $(\sle4)$-faces
are uncoloured. Observe that every nice colouring can
be greedily extended to an \col{9} of $G$, since
$\abs{{\mathcal C}(f)}\le8$ for each $(\sle4)$-face $f$, i.e.~$f$ has
at most $8$ forbidden colours.
Therefore, in the rest of the paper, we shall always suppose that such
faces are coloured at the very end. More
precisely, every time we consider a partial colouring of $G$, we uncolour
all $(\sle4)$-faces, and implicitly colour them at the very end of
the colouring procedure of $G$.
We make the following observation about nice colourings, which we rely on frequently.
\bigskip

\noindent
\textbf{Observation.} Let $e$ be an edge of $G$ incident to two faces
$f$ and $f'$.
There exists a nice colouring $\lambda$ of $G-e$, and hence
a partial \col{9} of $G$ in which only $e$ and $f$ are uncoloured.
Moreover, if $f$ is an $(\sle4)$-face, then it suffices to properly
colour the edge $e$ with
a colour from $\{1,2,\ldots,9\}$ to extend $\lambda$ to a nice
colouring of $G$.
\bigskip

The following lemma implies the reducibility of configuration A0.  We require the stronger form as it is necessary for later arguments, in particular, for the reducibility of configurations A1--A3.

\begin{lemma}
\label{cutvertex}
Let $v$ be a vertex of $G$ with neighbours $v_0,v_1,\ldots,v_{d-1}$ in clockwise order.
If $v$ is a cut-vertex of $G$, then no component $C$ of $G-v$ is such that
the neighbourhood of $v$ in $C$ is contained in $\{v_i,v_{i+1}\}$ for
some $i\in\{0,1,\ldots,d-1\}$, where the index $i$ is taken modulo $d$.
\end{lemma}

\begin{proof}
Suppose on the contrary that $C$ is a component of $G-v$ such
that the neighbourhood of $v$ in $C$ is contained in, say,
$\{v_0,v_1\}$.

First, assume that the neighbourhood of $v$ in $C$ is $\{v_0,v_1\}$. Then $G$ is the edge-disjoint union of
two plane graphs $G_1=(C\cup\{v\},E_1)$ and $G_2=(V\setminus C,E_2)$.
The outer face $f_1$ of $G_1$
corresponds to a face $f_2$ of $G_2$. By the minimality of $G$, the
graph $G_i$ has an \col{9} $\lambda_i$ for $i\in\{1,2\}$.
Since both $vv_0$ and $vv_1$ are incident in $G_1$ to $f_1$, we may assume that $\lambda_1(f_1)=1$, $\lambda_1(vv_0)=8$ and $\lambda_1(vv_1)=9$.
Regarding $\lambda_2$, we may
assume that $\lambda_2(f_2)=1$. Furthermore, up to permuting the
colours, we can also assume that the colours of the edges of $G_2$
incident to $v$ are contained in $\{1,2,\ldots,7\}$, since there are at
most $6$ such edges.

We now define an \col{9} $\lambda$ of $G$
as follows. For every edge $e$ of $G$, set $\lambda(e):=\lambda_1(e)$ if
$e\in E_1$ and $\lambda(e):=\lambda_2(e)$ if $e\in E_2$.
To colour the faces of $G$, let $f$ be the face of $G$ incident to both
$vv_0$ and $vv_{d-1}$. (Note that there is only one such face, since
otherwise $v$ would have degree $2$, which would be a contradiction.)
Now observe that there is a natural one-to-one
correspondence between the faces of $G_1$ and a subset $F_1$ of
the face set $F$ of $G$ that maps $f_1$ to $f$.
Similarly, there is a natural one-to-one
correspondence between the faces of $G$ and a subset $F_2$ of $F$ that
maps $f_2$ to $f$. Note that $F_1\cap F_2=\{f\}$. Now, we
can colour every face $f\in F_i$ using $\lambda_i$. This is well defined
since $\lambda_1(f_1)=\lambda_2(f_2)=1$.

Let us check that $\lambda$ is proper.
Two adjacent edges of $G$ are assigned different colours. Indeed, if the two edges
belong to $E_i$ for some $i\in\{1,2\}$, then it comes from the fact that
$\lambda_i$ is an \col{9} of $G_i$. Otherwise, both edges are incident
with $v$, and one is in $G_1$ and the other in $G_2$. The former is
coloured either $8$ or $9$, and the latter with a colour of
$\{1,2,\ldots,7\}$ by the choice of $\lambda_1$ and $\lambda_2$.
Two adjacent faces in $G$
necessarily correspond to two adjacent faces in $G_1$ or $G_2$, and
hence are assigned different colours. Last, let $g$ be a face of $G$ and
$e$ an edge incident to $g$ in $G$. If $g\neq f$, then $g$ and $e$ are
incident in $G_1$ or $G_2$, and hence coloured differently. Otherwise
$e$ is incident to $f_i$ in $G_i$ for some $i\in\{1,2\}$, and hence
$\lambda(e)=\lambda_i(e)\neq\lambda_i(f_i)=1=\lambda(f)$.

The case where the neighbourhood of $v$ in $C$ is $\{v_0\}$, i.e.\ $vv_0$ is a cut-edge, is dealt with
in the very same way so we omit it.
\end{proof}

The next lemma shows the reducibility of configurations B1 and C1.
\begin{lemma}
\label{uv1112}
Let $uv$ be an edge of $G$, and let $s\in\{1,2\}$ be the number of
$(\sle4)$-faces incident to $uv$. Then
$\deg(u)+\deg(v)\ge9+s$.
\end{lemma}

\begin{proof}
Suppose on the contrary that $\deg(u)+\deg(v)\le8+s$.
Let $f$ and $f'$ be the two faces incident to $uv$.

Without loss of generality assume that
$f$ is an $(\sle4)$-face. By the minimality of $G$, the graph $G-uv$ has a nice colouring
$\lambda$. Let $f''$ be the face of $G-uv$ corresponding to the union of
the two faces $f$ and $f'$ of $G$ after having removed the edge $uv$.
We obtain a partial \col{9} of $G$ in which only $uv,f$ and the
$(\sle4)$-faces are uncoloured by just assigning the colour $\lambda(f'')$ to $f'$, and
keeping all the other assignments.

Consequently, $\abs{{\mathcal C}(uv)}\le\deg(u)+\deg(v)-2+2-s\le8$.
Hence, we can properly colour the edge $uv$, thereby obtaining a nice
colouring of $G$; a contradiction.
\end{proof}

In light of Lemma~\ref{uv1112}, we make the following definition and observation.
An edge $uv$ of $G$ is called \emph{tight} if $\deg(u)+\deg(v)-s=9$, where $s\in\{1,2\}$ is the number of $(\sle4)$-faces incident to $uv$.
\bigskip

\noindent
\textbf{Observation.}
Assume that $c$ is an \col{9} of $G$ in
which only $uv$ and the $(\sle4)$-faces are uncoloured. Let $S$ be the
(possibly empty) set of colours assigned by $c$ to the $(\sge5)$-faces
incident to $uv$. If $uv$ is tight, then the sets ${\mathcal E}(u)$,
${\mathcal E}(v)$ and $S$ are pairwise disjoint, and ${\mathcal C}(uv)={\mathcal E}(u)\cup{\mathcal E}(v)\cup S=\{1,\ldots,9\}$.
\bigskip

The reducibility of configurations B2, B3, B4, C2 and C3 follows from the next lemma.
\begin{lemma}
\label{lem:uvw7}
Let $uvw$ be a triangle of $G$ such that $\deg(u)+\deg(v)=10+s$, where
$s\in\{0,1\}$ is the number of $(\sle4)$-faces distinct from $uvw$
incident to $uv$, and let $t\in\{0,1,2\}$ be the number of
$(\sle4)$-faces distinct from $uvw$ incident to $uw$ or $vw$.
Then $\deg(w)\ge7+t$.
\end{lemma}
\begin{proof}
As we pointed out, there exists a partial \col{9} $c$ of $G$ in which
only $uv$ and the $(\sle4)$-faces are left uncoloured.
Let $\alpha_{uv}$, $\alpha_{vw}$ and $\alpha_{uw}$ be the colours, if any, assigned to the
$(\sge5)$-faces incident to $uv$, $vw$ and $uw$, respectively.
Since the edge $uv$ is tight, ${\mathcal E}(u)$, ${\mathcal E}(v)$ and
$\{\alpha_{uv}\}$ form a partition of $\{1,2,\ldots,9\}$. Thus, if there is a colour
$\xi\in{\mathcal E}(u)\cup\{\alpha_{uv}\}$ that is not in
${\mathcal E}(w)\cup\{\alpha_{vw}\}$, then we can colour $uv$ with
$c(vw)$ and next recolour $vw$ with $\xi$ to obtain a nice colouring
of $G$. We deduce that
${\mathcal E}(u)\cup\{\alpha_{uv}\}\subseteq{\mathcal E}(w)\cup\{\alpha_{vw}\}$.
Similarly, ${\mathcal E}(v)\cup\{\alpha_{uv}\}\subseteq{\mathcal E}(w)\cup\{\alpha_{uw}\}$.
Hence, $\deg(w)+2-t \ge 9$, so $\deg(w)\ge7+t$, as required.
\end{proof}

The following verifies that configuration A1 is reducible. The lemma is also needed for showing the reducibility of configurations A3, C4 and C5.
\begin{lemma}
\label{lem:a1}
Let $u,v,w$ be vertices of $G$ with 
$\deg(v)=2$. Then $uvw$ is not a face of $G$.
\end{lemma}
\begin{proof}
Suppose on the contrary that $uvw$ is a face of $G$.
There exists a
nice colouring $c$ of $G-uv$.
Note that the face $f_1$ of $G$ other than $uvw$ that is incident to both $uv$ and $vw$ must be an $(\sge5)$-face, or else $\abs{{\mathcal C}(uv)} \le 8$ and we can immediately extend $c$ to a nice colouring of $G$.
Note that $f_1$ is distinct from the face $f_2$ of $G$ other than $uvw$ that is incident to $uw$; otherwise, one of $u$ or $w$ is a cut-vertex of a type forbidden by Lemma~\ref{cutvertex}.
Let $\beta_{uw}$ be the colour, if any, assigned by $c$ to $f_2$.

Observe that $c(f_1) \notin \{\beta_{uw}, c(uw)\}$.
Indeed, since $uv$ is tight, the sets ${\mathcal E}(u)$, $\{c(vw)\}$ and $\{c(f_1)\}$
are pairwise disjoint.
Since $vw$ is tight, we
deduce that $c(f_1)\notin{\mathcal E}(w)$, for otherwise we could colour
$uv$ with $c(vw)$ and next recolour $vw$ with a colour from
$\{1,\ldots,9\}\setminus{\mathcal E}(w)$.
Hence $c(f_1) \notin {\mathcal E}(u) \cup {\mathcal E}(w) \cup
\{\beta_{uw}\}$, so that colouring $uv$ with $c(uw)$ and next
recolouring $uw$ with $c(f_1)$ yields a nice colouring of $G$; a contradiction.
\end{proof}

The next lemma will be used to show the reducibility of configurations A3 and E4.
\begin{lemma}\label{smalladj}
Let $v$ be a $2$-vertex of $G$, and let $u$ and $w$ be its two neighbours.
If $\deg(u)\le6$, then $u$ and $w$ are adjacent in $G$.
\end{lemma}
\begin{proof}
Suppose on the contrary that $u$ and $w$ are not adjacent in $G$.
Then, the graph $G'$ obtained by contracting the edge
$uv$ is planar, simple and has maximum degree at most $8$.
By the minimality of $G$, let $\lambda$ be a nice colouring of $G'$.
Let $g$ and $g'$ be the faces of $G'$ corresponding
to the contracted faces $f$ and $f'$ of $G$, respectively.
We obtain a partial \col{9} of $G$ in which only $uv$
is uncoloured by assigning the colour $\lambda(g)$ to $f$, the colour $\lambda(g')$ to $f'$, and
keeping all the other assignments.

Now, $\abs{{\mathcal C}(uv)}\le\deg(u)+\deg(v)-2+2\le8$.
Consequently, we can properly colour the edge $uv$ to obtain a nice
colouring of $G$; a contradiction.
\end{proof}

We now deduce the reducibility of configuration A3.
\begin{corollary}
Let $v$ be a $2$-vertex of $G$ and let $u$ and $w$ be its two neighbours. If $\deg(u)=3$, then $\deg(w)\ge6$.
\end{corollary}
\begin{proof}
Suppose on the contrary that $u$ has degree $3$ and $w$ has degree at most $5$. Lemma~\ref{smalladj} implies that
$u$ and $w$ are adjacent. Note that $uvw$ cannot be a face by Lemma~\ref{lem:a1}.
Let $u'$ be the neighbour of $u$ besides $v$ and $w$.  By Jordan's curve theorem, the curve $uvw$ splits the plane
into two parts, $\mathscr{I}$ and $\mathscr{O}$ with $u'\in \mathscr{O}$. First, note that
$w$ has a neighbour in $\mathscr{O}$, for otherwise $u$ would be a cut-vertex that contradicts Lemma~\ref{cutvertex}.
Moreover, Lemma~\ref{lem:a1} implies that $w$ has a neighbour in $\mathscr{I}$. Consequently, $w$ has either one or two neighbours in
$\mathscr{I}$, and hence $w$ is a cut-vertex that contradicts Lemma~\ref{cutvertex}.
\end{proof}

The following demonstrates that configuration A2 is reducible.
\begin{lemma}
\label{lem:a2}
Let $u,v,w,x$ be vertices of $G$ with 
$\deg(v)=2$ and $\deg(x)\le3$. Then $uvwx$ is not a face
of $G$.
\end{lemma}
\begin{proof}
Suppose on the contrary that $uvwx$ is a face of $G$.
There exists a
partial \col{9} $c$ of $G$ in which only $uv$ and the
$(\sle4)$-faces are uncoloured. Let $\alpha$ be the colour, if any, assigned
to the $(\sge5)$-face incident to both $uv$ and $vw$, and let $\beta_{ux}$ and $\beta_{wx}$ be the colours, if any,
assigned to the $(\sge5)$-faces incident to $ux$ and $wx$, respectively.

By Lemma~\ref{cutvertex},
observe that $\alpha \notin \{\beta_{ux}, \beta_{wx},
c(ux), c(wx)\}$.
Since $uv$ is tight, the sets ${\mathcal E}(u)$, $\{c(vw)\}$ and $\{\alpha\}$
are pairwise disjoint.
Since $vw$ is tight, we
deduce that $\alpha\notin{\mathcal E}(w)$, for otherwise we could colour
$uv$ with $c(vw)$ and next recolour $vw$ with a colour from
$\{1,\ldots,9\}\setminus{\mathcal E}(w)$.
Hence $\alpha \notin {\mathcal E}(u) \cup {\mathcal E}(w) \cup
\{\beta_{ux},\beta_{wx}\}$.

Let $x'$ be the vertex adjacent to $x$
distinct from $u$ and $w$. We must have $c(xx')=\alpha$, otherwise we could
colour $uv$ with $c(ux)$ and next recolour $ux$ with $\alpha$. Since $\beta_{ux}
\neq \beta_{wx}$, at least one of $\beta_{ux}$ and $\beta_{wx}$ is distinct from
$c(vw)$. Observing that we can colour $uv$ with $c(vw)$ and next
uncolour $vw$, we may assume without loss of generality that $\beta_{wx} \neq c(vw)$.
As a result, colouring $uv$ with $c(vw)$, and next swapping the colours
of $vw$ and $xw$ yields a
nice colouring of $G$; a contradiction.
\end{proof}

The following verifies that configurations C4 and C5 are reducible.

\begin{lemma}
Let $uvw$ and $vwx$ be triangles of $G$ such that $wx$ is incident to
two $(\sle 4)$-faces.
\begin{enumerate}
\item \label{lem:c4}
At least one of $u$ and $x$ has degree at least $4$.
\item \label{lem:c5}
If $uv$ is tight, then $\deg(v)+\deg(x)\ge12$.
\end{enumerate}
\end{lemma}

\begin{proof}
\ref{lem:c4}. Suppose on the contrary that both $u$ and $x$ have degree
less than $4$. Then both have degree $3$ by Lemma~\ref{lem:a1}.
Let $u'$ (respectively $x'$) be the neighbour of $u$ (respectively $x$) distinct from $v$ and $w$.
Let $c$ be a
partial \col{9} of $G$ in which only $wx$ and the
$(\sle4)$-faces are uncoloured.
Let $\alpha_{uv}$, $\alpha_{uw}$ and $\alpha_{vx}$ be
the colours, if any, assigned to the $(\sge5)$-faces incident to $uv$, $uw$ and $vx$, respectively.

Since the edge $wx$ is tight, the sets ${\mathcal E}(w)$ and ${\mathcal E}(x)$ are
disjoint.
Hence $c(xx')\in {\mathcal E}(v)$, otherwise we could colour $wx$
with $c(vw)$ and recolour $vw$ with $c(xx')$.

We first assert that $\alpha_{vx}\neq c(vw)$. Otherwise,
$\abs{{\mathcal C}(vx)}=\abs{{\mathcal E}(v)}\le8$ and
there exists
$\xi\in\{1,2,\ldots,9\}\setminus{\mathcal C}(vx)$.
Now, colouring $wx$ with $c(vx)$ and recolouring $vx$ with $\xi$ yields
a nice colouring of $G$; a contradiction.
Consequently, we can safely swap the colours of $vw$ and $vx$, if necessary.

Our next assertion is that $\{c(uu'),\alpha_{uw}\}=\{c(vw),c(vx)\}$.
For, if $c(vx)\notin\{c(uu'),\alpha_{uw}\}$, we can colour $wx$ with
$c(uw)$ and recolour $uw$ with $c(vx)$; a contradiction. The same argument after swapping
the colours of $vw$ and $vx$ shows that
$c(vw)\in\{c(uu'),\alpha_{uw}\}$. Thus, up to swapping the colours of
$vw$ and $vx$, we may assume that $c(vx)=\alpha_{uw}$.

Let us recolour $uv$ with $c(vx)$, colour $wx$ with $c(vx)$ and
uncolour $vx$. The obtained colouring is proper,
since $\alpha_{uv}\neq\alpha_{uw}=c(vx)$ and ${\mathcal E}(w)\cap{\mathcal E}(x)=\emptyset$.
Now, if $vx$ cannot be coloured greedily, then for the obtained
colouring
${\mathcal E}(v)\cup{\mathcal E}(x)\cup\{\alpha_{vx}\}=\{1,2,\ldots,9\}$.
But then, since there are at most ten in the set of edges incident to $v$ or $x$, two of which ($uv$ and $wx$) have the same colour and one of which is uncoloured, it follows that
$c(xx')\notin{\mathcal E}(v)\cup{\mathcal E}(w)$.  Now we
can colour $vx$ with $c(vw)$ and colour $vw$ with $c(xx')$ to obtain a
nice colouring of $G$; a contradiction.
\bigskip

\ref{lem:c5}. Suppose on the contrary that $uv$ is tight and $\deg(v)+\deg(x)=11$.
Let $c$ be
a partial \col{9} of $G$ in which only $uv$ and the
$(\sle4)$-faces are uncoloured.
Let $\alpha_{uv}$, $\alpha_{uw}$ and $\alpha_{vx}$ be
the colours, if any, assigned to the $(\sge5)$-faces incident to $uv$, $uw$ and $vx$, respectively.
Since the edge $uv$ is tight, the sets
${\mathcal E}(u)$, ${\mathcal E}(v)$ and $\{\alpha_{uv}\}$ are pairwise disjoint.

Let $\xi$ be the colour in $\{1,\ldots,9\}\setminus{\mathcal E}(w)$ (unique since we can assume that $\deg(w)=8$ without loss of generality). Then $\xi \in {\mathcal E}(v)$,
otherwise we could colour $uv$ with $c(vw)$ and recolour $vw$ with $\xi$.
It follows that $\xi \notin {\mathcal E}(u)$. Therefore, $\alpha_{uw}=\xi$,
otherwise we could colour $uv$ with $c(uw)$ and recolour $uw$ with $\xi$.
Thus, the colours of $uw$ and $vw$ may be exchanged, if necessary.

Let us show that ${\mathcal E}(u)\cup\{\alpha_{uv},c(vw)\} \subseteq
{\mathcal E}(x)\cup\{\alpha_{vx}\}$.
First, if there is a colour
$\gamma \in
{\mathcal E}(u)\cup\{\alpha_{uv}\}$ that is not in
${\mathcal E}(x)\cup\{\alpha_{vx}\}$, then we can recolour $vx$ with
$\gamma$ and then colour
$uv$ with $c(vx)$ to obtain a nice colouring of $G$, which is a contradiction. Similarly, by exchanging the colours of $uw$ and $vw$, we conclude that $c(vw) \in {\mathcal E}(x)\cup\{\alpha_{vx}\}$.

Since $uv$ is tight and $\deg(v)+\deg(x)=11$, 
we deduce that ${\mathcal E}(x)\cup\{\alpha_{vx}\}={\mathcal
E}(u)\cup\{\alpha_{uv},c(vw),c(vx)\}$.
(Indeed, $\abs{{\mathcal E}(x)\cup\{\alpha_{vx}\}}\le\deg(x)+1=12-\deg(v)$,
and $\abs{{\mathcal E}(u)\cup\{\alpha_{uv}\}}=9-(\deg(v)-1)=10-\deg(v)$.)
In particular, $\alpha_{vx}\neq
c(wx)$ and $\xi \notin {\mathcal E}(x) \setminus \{c(vx)\}$. Now, colour $uv$ with $c(vx)$, and then
recolour $vx$ with $c(wx)$ and $wx$ with $\xi$ to obtain a
nice colouring of $G$; a contradiction.
\end{proof}

The next lemma implies that configurations D1--D4 are reducible.

\begin{lemma}\label{lem.D1-4}
Let $vwx$ be a triangle of $G$ and $u$ a neighbour of $v$ distinct from
$x$ and $w$.
If
$vx$ is incident to two $(\sle4)$-faces,
then either $uv$ or $vw$ is not tight.
\end{lemma}

\begin{proof}
Suppose on the contrary that both $uv$ and $vw$ are tight. Let $c$ be a partial
\col{9} of $G$ in which only $vw$ and the $(\sle4)$-faces
are left uncoloured. Let $\alpha$ be the colour, if any, assigned to the
$(\sge5)$-face incident to $vw$.
Since $vw$ is tight, we know that the sets ${\mathcal E}(v)$,
${\mathcal E}(w)$ and $\{\alpha\}$ form a partition of
$\{1,2,\ldots,9\}$. In particular,
$c(vx)\notin{\mathcal E}(w)$ and $c(wx)\notin{\mathcal E}(v)$.

If an edge $e$ that is adjacent to $vw$
could be properly recoloured with a colour $\xi$, then colouring $vw$ with
$c(e)$ and recolouring $e$ with $\xi$ would yield a nice colouring of $G$;
a contradiction. Applying this to $vx$ yields that
${\mathcal E}(w)\cup\{\alpha\}\subseteq{\mathcal E}(x)$, since
${\mathcal C}(vx)={\mathcal E}(x)\cup{\mathcal E}(v)$, and as we noted above
$\{1,\ldots,9\}\setminus{\mathcal E}(v)={\mathcal E}(w)\cup\{\alpha\}$.
Applying the same remark to $wx$, we obtain
${\mathcal E}(v)\cup\{\alpha\}\subseteq{\mathcal E}(x)\cup\{\beta\}$,
where $\beta$ is the colour, if any, assigned to the $(\sge5)$-face
incident to $wx$.

Since
$9=\abs{{\mathcal E}(v)\cup{\mathcal E}(w)\cup\{\alpha\}}\le\abs{{\mathcal E}(x)\cup
\{\beta\}}\le9$,
we deduce that $\beta\notin{\mathcal E}(x)$.
Therefore, we can safely swap the colours
of $vx$ and $wx$ if needed (recalling that
${\mathcal E}(v)\cap{\mathcal E}(w)=\emptyset$).

Let $S$ be the set of colours of the $(\sge5)$-faces incident to $uv$. Thus,
$\abs{S}=2-s$ where $s$ is the number of $(\sle4)$-faces incident to $uv$.
Again, we apply the same arguments as above to $uv$: since $uv$ cannot be
recoloured, we deduce that
${\mathcal E}(u)\cup{\mathcal E}(v)\cup S=\{1,2,\ldots,9\}$.
But
$\abs{{\mathcal E}(u)\cup{\mathcal E}(v)\cup
S}\le\deg(u)-1+\deg(v)-1+2-s=\deg(u)+\deg(v)-s\le9$ since $uv$ is tight
and $vw$ is uncoloured.
Consequently,
${\mathcal E}(u)$, ${\mathcal E}(v)$ and $S$ are pairwise disjoint.
In particular,
$c(vx)\notin{\mathcal E}(u)\cup S$. As a result, colouring $vw$
with $c(uv)$, then recolouring $uv$ with $c(vx)$ and finally swapping
the colours of $vx$ and $wx$ yields a nice colouring of $G$; a
contradiction.
\end{proof}

The next lemma implies that configurations E1 and E2 are reducible.

\begin{lemma}\label{lem.E}
Let $v$ be an $8$-vertex of $G$ with neighbours $v_0,v_1,\ldots,v_7$ in
anti-clockwise order. Assume that $v_iv_{i+1}$ is an edge for
$i\in\{0,1,2,3\}$, and that $v_1$ an
$(\sle4)$-vertex. If $v_0$ is an $(\sle6)$-vertex or $vv_0$ is adjacent to two $(\sle4)$-faces, then $v_3$ is an $(\sge4)$-vertex.
\end{lemma}

\begin{proof}
Suppose on the contrary that $v_3$ is a $3$-vertex.
By the minimality of $G$, the graph $G-vv_3$ has a nice colouring and hence $G$ has a partial \col{9} $c$ in which only $vv_3$ and the $(\sle4)$-faces are left uncoloured.
Since $vv_3$ is tight, we deduce that $\abs{{\mathcal E}(v)\cup{\mathcal E}(v_3)}=9$ and ${\mathcal E}(v)\cap{\mathcal E}(v_3)=\emptyset$.

Let $\alpha$ be the colour, if any, of the $(\sge5)$-face incident with both $v_2v_3$ and $v_3v_4$.
If $v_2v_3$ can be recoloured with a colour $\xi$, then colouring $vv_3$ with $c(v_2v_3)$ and then $v_2v_3$ with $\xi$ would yield a nice colouring of $G$; a contradiction.
Thus, ${\mathcal E}(v)\subseteq{\mathcal E}(v_2)\cup\{\alpha\}$.

Let $j\in\{1,2\}$. If there exists a
colour $\xi\in{\mathcal E}(v_3)\setminus{\mathcal E}(v_j)$, then
colouring $vv_3$ with $c(vv_j)$ and then $vv_j$ with $\xi$
yields a nice colouring of $G$ (recalling that ${\mathcal E}(v_3)$ and
${\mathcal E}(v)$ are disjoint).
Therefore,
${\mathcal E}(v_3)\subseteq{\mathcal E}(v_j)$ for $j\in\{1,2\}$.
Letting $\gamma$ be the colour, if any, of the $(\sge5)$-face incident to $vv_0$ we similarly find that ${\mathcal E}(v_3)\subseteq{\mathcal E}(v_0)\cup\{\gamma\}$.

Since ${\mathcal E}(v_2)\cup\{\alpha\}\supseteq{\mathcal E}(v)\cup{\mathcal E}(v_3)=\{1,2,\ldots,9\}$ and $\abs{{\mathcal E}(v_2)\cup\{\alpha\}}\le9$, it follows that $\alpha \ne c(vv_2)$.
As
${\mathcal E}(v)\cap{\mathcal E}(v_3)=\emptyset$, this implies that the
colours of $vv_2$ and $v_2v_3$ can be freely swapped.
By doing so, we can conclude that ${\mathcal E}(v_3)\cup\{c(vv_2)\}\subseteq{\mathcal E}(v_j)$ for $j\in\{1,2\}$ and ${\mathcal E}(v_3)\cup\{c(vv_2)\}\subseteq{\mathcal E}(v_0)\cup\{\gamma\}$.

Since $\deg(v_1)=4$, we find that ${\mathcal E}(v_1)=\{c(vv_1),c(vv_2)\}\cup{\mathcal E}(v_3)$.  Furthermore, by swapping the colours of $vv_2$ and $v_2v_3$ if necessary, we may assume that $c(v_0v_1)\in{\mathcal E}(v_3)$.
Now,
if $v_0v_1$ could be recoloured with a colour $\xi$, then colouring
$vv_3$ with $c(vv_1)$, then $vv_1$ with $c(v_0v_1)$ and then $v_0v_1$ with $\xi$ would yield a nice colouring of $G$.
Thus, letting $\beta$ be the colour, if any, of the $(\sge5)$-face incident to
$v_0v_1$ we obtain ${\mathcal E}(v_0)\cup{\mathcal E}(v_1)\cup\{\beta\}=\{1,2,\ldots,9\}$.

Let us partition our analysis now based on if $v_0$ is an $(\sle6)$-vertex or if $vv_0$ is adjacent to two $(\sle4)$-faces.

Suppose we are in the former case.
Since ${\mathcal E}(v_3)\cup\{c(vv_2)\}\subseteq({\mathcal E}(v_0)\cup\{\gamma\})\cap{\mathcal E}(v_1)$,
we deduce that
$\abs{{\mathcal E}(v_0)\cup{\mathcal E}(v_1)}\le\deg(v_0)+\deg(v_1)-2\le8$.
Consequently, $\beta\neq c(vv_1)$ and $c(vv_1)\notin{\mathcal E}(v_0)$.
In particular, the colours of $vv_1$ and $v_0v_1$ can safely be
swapped if needed.
As a result, colouring $vv_3$ with $c(vv_1)$ and then swapping the
colours of $vv_1$ and $v_0v_1$ yields a nice colouring of $G$; a
contradiction.

Now suppose we are in the latter case.  Then there is no colour $\gamma$.
For $j\in\{0,1\}$, it cannot be that $c(vv_j) \in {\mathcal E}(v_{1-j})$ (and hence $c(vv_j) \in {\mathcal E}(v_0)\cap{\mathcal E}(v_1)$).  Otherwise, we would have,
using ${\mathcal E}(v_3)\cup\{c(vv_2)\}\subseteq {\mathcal E}(v_0)\cap{\mathcal E}(v_1)$,
that
$\abs{{\mathcal E}(v_0)\cup{\mathcal E}(v_1)}\le\deg(v_0)+\deg(v_1)-4\le8$, in which case, recolouring as we did in the last paragraph, we would reach a contradiction. However, for some $j\in\{0,1\}$, we must have $\beta\ne c(vv_j)$, and so the colours of $vv_j$ and $v_0v_1$ can be swapped safely.  Thus, colouring $vv_3$ with $c(vv_j)$ and then swapping the colours of $vv_j$ and $v_0v_1$ yields a nice colouring of $G$; a contradiction.
\end{proof}

In the following lemma, we show that configuration E3 is reducible.

\begin{lemma}\label{lem.E3}
Let $v$ be a triangulated $8$-vertex of $G$ with neighbours $v_0,v_1,\ldots,v_7$ in
anti-clockwise order. If $v_0$ is a $3$-vertex, then every vertex $v_i$ with
$i\ne0$ has degree at least $5$.
\end{lemma}

\begin{proof}
Suppose on the contrary that $v_j$ is an $(\sle4)$-vertex with
$j\in\{1,\ldots,7\}$. First, note that $j\notin\{1,7\}$ by Lemma~\ref{lem:uvw7} (the reducibility of configuration B2, in particular).
By the minimality of $G$, the graph $G-vv_0$ has a nice colouring, and
hence the graph $G$ has a partial \col{9} in which only
$vv_0$ and the $(\sle4)$-faces are left uncoloured.
Since $vv_0$ is tight and incident to two triangles,
we infer that $\abs{{\mathcal E}(v)\cup{\mathcal E}(v_0)}=9$ and
${\mathcal E}(v)\cap{\mathcal E}(v_0)=\emptyset$.

Note that ${\mathcal E}(v_0)\subset{\mathcal E}(v_i)$ for $i\ne0$,
for otherwise we could colour $vv_0$ with $c(vv_i)$ and then
recolour $vv_i$ with a colour in ${\mathcal E}(v_0)\setminus{\mathcal E}(v_i)$
to obtain a nice colouring of $G$ (recalling that
${\mathcal E}(v)\cap{\mathcal E}(v_0)=\emptyset$).
Since $\{c(vv_j)\}\cup{\mathcal E}(v_0)\subseteq{\mathcal E}(v_j)$ and
$\deg(v_j)\le4$, we deduce that one of $c(vv_1)$ and $c(vv_7)$ does not belong to
${\mathcal E}(v_j)$, say $c(vv_7)$.

Let $\alpha$ be the colour of the face incident to both $v_0v_1$ and
$v_0v_7$. We prove
that $\alpha\neq c(vv_7)$.
Indeed, suppose on the contrary that $\alpha=c(vv_7)$. Then, there exists a
colour $\xi$ that does not belong to
${\mathcal E}(v_7)\cup\{\alpha\}={\mathcal E}(v_7)$, since
$\deg(v_7)\le8$. As ${\mathcal E}(v_0)\subset{\mathcal E}(v_7)$, we deduce
that $\xi\notin{\mathcal E}(v_0)\cup{\mathcal E}(v_7)\cup\{\alpha\}$.
Therefore, colouring $vv_0$ with $c(v_0v_7)$ and then
$v_0v_7$ with $\xi$ yields a nice colouring of $G$; a
contradiction.
Hence, $\alpha\neq c(vv_7)$.
Consequently, we can freely
swap the colours of $vv_7$ and $v_0v_7$.
Now, colouring $vv_0$ with $c(vv_j)$, then recolouring $vv_j$ with
$c(vv_7)$ and last swapping the colours of $vv_7$ and $v_0v_7$
yields a nice colouring of $G$; a contradiction.
\end{proof}

The next lemma implies that configuration E4 is reducible.
\begin{lemma}\label{lem-exceptional}
Let $v$ be a $6$-vertex of $G$ with neighbours $u$ and $w$, and suppose $x\ne v$ is a neighbour of $w$.  Suppose $w$ is a $2$-vertex and $x$ is a $3$-vertex. Assume that $uv$ is adjacent to an $(\sle4)$-face. Then $u$ is an $(\sge6)$-vertex.
\end{lemma}
\begin{proof}
First of all note that, due to Lemma~\ref{smalladj}, $v$ and $x$ are adjacent in $G$.
Suppose on the contrary that $u$ is an $(\sle5)$-vertex. By the minimality of $G$, the graph $G-uv$ has a nice colouring and hence $G$ has a partial \col{9} in which only $uv$ and the $(\sle4)$-faces are left uncoloured. Let us further uncolour $vw$ and $vx$. Now, $\abs{{\mathcal C}(uv)}\le\deg(u)-1+\deg(v)-1-2+1\le8$, so we may properly colour $uv$. It remains to colour $vw$ and $vx$.  Next, since $vw$ was uncoloured, we see that $\abs{{\mathcal C}(vx)}\le\deg(v)-1+\deg(x)-1-1+2=8$, so we may properly colour $vx$. Finally, we consider $vw$ and notice that $\abs{{\mathcal C}(vw)}\le\deg(v)-1+\deg(w)-1+2=8$, which does not prevent us from properly colouring $vw$. We have thereby obtained a nice colouring of $G$; a contradiction.
\end{proof}

It remains to prove Lemmas~\ref{lem-new} and~\ref{lem-newnew}.
\begin{proof}[Proof of Lemma~\ref{lem-new}]
By the minimality of $G$, the (proper) subgraph $G'$ formed from $G$ by deleting all loose edges incident to $f$ has a nice colouring.
To extend this to a nice colouring of $G$, it would suffice to properly colour $f$, as every loose edge on $f$ can then be greedily coloured. Indeed, a loose edge $uv$ on $f$ is incident to at most $\deg(u)-1+\deg(v)-1\le6$ other edges. Consequently, since $G$ cannot have a nice colouring, we conclude that $f$ is incident or adjacent to elements of all nine colours. Now, $f$ is adjacent to at most $d-q$ other faces, and incident to $d-x$ coloured edges. Therefore, $d-q+d-x\ge9$, as asserted.
\end{proof}

\begin{proof}[Proof of Lemma~\ref{lem-newnew}]
That $u' = v'$ follows directly from Lemma~\ref{smalladj}.  Note that by Lemma~\ref{lem:a1}, $uvu'$ is not a face.  By Jordan's curve theorem, the curve $uvu'$ splits the plane
into two parts, $\mathscr{I}$ and $\mathscr{O}$.  Then $u'$ must have three neighbours in $\mathscr{I}$ and three neighbours in $\mathscr{O}$, or else it would be a cut vertex of a type forbidden by Lemma~\ref{cutvertex}.  This implies that $u'$ has degree $8$.
\end{proof}

\section*{Acknowledgements}
We thank Artem Pyatkin for pointing out an error in a previous version of this work.  We are grateful to the anonymous referees for their thorough reading and helpful suggestions.

This work was initiated during a visit by Ross Kang to Centre \'Emile Borel, Institut Henri Poincar\'e in Paris for the thematic term ``Statistical physics, combinatorics and probability: from discrete to continuous models'', supported by the {\em European Research Council} (ERC), grant ERC StG 208471 ExploreMaps.

\bibliographystyle{plain}
\bibliography{kss10}
\end{document}